\documentclass[a4paper]{amsart}

\usepackage{amsmath,amssymb,amscd}
\usepackage{extarrows}

\usepackage{tikz-cd}

\numberwithin{equation}{section}
\newtheorem{theorem}[equation]{Theorem}

\newtheorem{corollary}[equation]{Corollary}
\newtheorem{lemma}[equation]{Lemma}

\newtheorem{proposition}[equation]{Proposition}

\theoremstyle{remark}
\newtheorem{remark}[equation]{Remark}
\theoremstyle{definition}
\newtheorem{definition}[equation]{Definition}

\newtheorem{example}[equation]{Example}
\newtheorem{hypothesis}[equation]{Hypothesis}

\newcounter{proofsec}
\newcommand{\proofsec}{\noindent\textbf{\theproofsec.}~\stepcounter{proofsec}}

\renewcommand{\epsilon}{\varepsilon}

\newcommand{\Alg}{\ccat{Alg}}

\DeclareMathOperator{\Aut}{Aut}

\newcommand{\cat}{\mathcal}
\newcommand{\Cat}{\ccat{Cat}}

\newcommand{\ccat}{\mathrm}

\DeclareMathOperator*{\colim}{colim}
\DeclareMathOperator{\Conf}{Conf}

\newcommand{\widebar}{\overline}
\newcommand{\close}{\widebar}

\newcommand{\cover}{\mathcal}

\newcommand{\Dis}{\ccat{Disj}}

\newcommand{\Disk}{\ccat{Disk}}
\newcommand{\disj}{\sqcup}
\newcommand{\Disj}{\bigsqcup}
\renewcommand{\dot}{\bullet}

\newcommand{\Elu}{\mathbb{E}}

\DeclareMathOperator{\Emb}{Emb}

\renewcommand{\equiv}{\sim}

\newcommand{\equivwith}{\simeq}
\newcommand{\equivto}{\xrightarrow{\equiv}}
\newcommand{\extend}{\widebar}

\newcommand{\field}{\mathbb}

\newcommand{\R}{\field{R}}

\newcommand{\from}{\leftarrow}
\newcommand{\Fun}{\ccat{Fun}}

\newcommand{\id}{\mathrm{id}}

\newcommand{\intersect}{\cap}
\newcommand{\Intersect}{\bigcap}

\newcommand{\into}{\hookrightarrow}
\newcommand{\kara}{\varnothing}
\newcommand{\kore}{\textbf}

\DeclareMathOperator*{\laxcol}{lax\,colim}

\newcommand{\loc}{\mathrm{loc}}

\newcommand{\longequivto}{\xlongrightarrow{\equiv}}
\newcommand{\longfrom}{\longleftarrow}
\newcommand{\longto}{\longrightarrow}

\newcommand{\Man}{\ccat{Man}}
\newcommand{\man}{\ccat{man}}
\newcommand{\manc}{c\,\man}
\newcommand{\Map}{\mathrm{Map}}

\DeclareMathOperator{\Multimap}{Multimap}

\newcommand{\op}{\mathrm{op}}
\newcommand{\Open}{\ccat{Open}}

\newcommand{\parmo}{\extend}

\newcommand{\pr}{\mathrm{pr}}
\newcommand{\Pre}{\ccat{Pre}}

\newcommand*{\resto}[1]{\mathclose{|}_{#1}}

\newcommand{\Simp}{\mathbf{\Delta}}
\newcommand{\simp}{\Delta}

\newcommand{\sub}{\subset}
\DeclareMathOperator{\supp}{supp}

\newcommand{\tensor}{\otimes}
\newcommand{\Tensor}{\bigotimes}

\newcommand{\wreath}{\boldsymbol{\wr}}

\title{Descent properties of the topological chiral homology}
\author[Matsuoka, Takuo]{Takuo Matsuoka}
\address{Department of Mathematics, Northwestern University, Evanston,
  IL, USA (Not current)}
\email{motogeomtop@gmail.com}
\date{}

\subjclass[2010]{16E40
, 55N25
, 55N30
}
\keywords{factorization algebra, topological chiral homology,
  homotopical algebra}

\begin{document}
\setcounter{section}{-1}
\setcounter{equation}{-1}
\maketitle

\begin{abstract}
We study descent properties of Jacob Lurie's topological chiral
homology.
We prove that this homology theory satisfies descent for a factorizing
cover, as defined by Kevin Costello and Owen Gwilliam.
We also obtain a generalization of Lurie's approach to this homology
theory, which leads to a product formula for the infinity
$1$-category of factorization algebras, and its twisted generalization.
\end{abstract}

\section{Introduction}
\setcounter{subsection}{-1}
\setcounter{equation}{-1}
\subsection{Factorization algebra and the topological chiral homology}
\setcounter{subsubsection}{-1}

\subsubsection{}

Developing on the work of Lurie \cite[Chapter 5]{higher-alg} on the
\emph{topological chiral homology}, we study interesting counterparts
on manifolds of \emph{factorization algebras} defined by Beilinson and
Drinfeld on algebraic curves \cite{bd} (which were shown by them to be
equivalent to \emph{chiral algebras}, introduced in the same work).
Following some of the pioneers of the research of these objects on
manifolds, we call them \emph{factorization algebras}.

One motivation for studying factorization algebras on manifolds comes
from the central role which they play in quantum field theory,
generalizing the role of chiral algebras for conformal field theory.
Namely, observables of a quantum (or a classical) field theory form a
factorization algebra, and this is the structure in terms of which one
can rigorously understand quantization of a physical theory (in
perturbative sense) \cite{cg}, analogously to the deformation
quantization of the classical mechanics \cite{kontsevi}.

Factorization algebras are closely related to \emph{field theories} as
functors on a cobordism category, as introduced by Atiyah
\cite{atiyah} and Segal \cite{segal}.
We study \emph{locally constant} factorization algebras, which
correspond to \emph{topological} field theories.

A locally constant factorization algebra on the manifold $\R^n$ is
equivalent to what is known as an $E_n$-algebra, first introduced in
iterated loop space theory \cite{bv}.
$E_1$-algebra is an associative algebra, and an \emph{$E_n$-algebra}
can be inductively defined as an $E_{n-1}$-algebra with an additional
structure of an associative algebra commuting with the
$E_{n-1}$-structure.
A locally constant factorization algebra can be considered as a global
version of an $E_n$-algebra in a way analogous to how a chiral algebra
is a global version of a vertex operator algebra.
In particular, from any locally constant factorization algebra on an
$n$-dimensional manifold, one obtains an $E_n$-algebra around any
point by restricting the algebra to an open ball around the point.
This $E_n$-algebra is canonical up to a change of framing at the
point, and can be thought of as a local form of the
factorization algebra.

There is an issue that the notion of an $E_n$-algebra degenerates
(unless $n\le 1$) to that of a commutative algebra in a category whose
higher homotopical structure is degenerate.
Moreover, some further developments such as the theory of the Koszul
duality for factorization algebras \cite{poincare} requires a nice
higher homotopical structure in order to lead to fruitful results,
even on the manifold $\R^1$.
These issues force us to work in a homotopical setting.
In order to work in such a setting, we use the convenient language of
higher category theory.
(For the main body, note our conventions stated in Section
\ref{sec:terminology-notation}, which do not apply in this
introduction.)
We just remark here that associativity of an algebra in such a
setting means a data for homotopy coherent associativity (which in
particular is a structure rather than a property).

\subsubsection{}
In this work, we study from the point of view that a factorization
algebra is a generalization of a sheaf on a manifold (the term
``locally constant'' comes from this point of view).
It takes values in a symmetric monoidal infinity $1$-category.
A \kore{prealgebra} on a manifold $M$ is
a covariant functor $A$ on the poset of open subsets of $M$, for which
we have $A(U\disj V)\equivwith A(U)\tensor A(V)$ for disjoint open
subsets $U,\: V\sub M$, in a coherent way.
(Covariance is chosen for consistency of the terminology with the
intuition.)
$A$ is a \kore{factorization algebra} if it satisfies a suitable
gluing condition generalizing that for a sheaf.
Indeed, a locally constant cosheaf is a locally constant factorization
algebra with respect to the monoidal structure given by the coproduct.

The gluing condition of the factorization algebra of observables of a
physical theory reflects locality of the theory.
In Atiyah--Segal framework, the same property corresponds to
possibility of extending the functor on cobordisms to higher
codimensional manifolds.
A theory is \emph{fully extended} if it is extended to highest
codimensional manifolds, namely, to points.
The cobordism hypothesis of Baez--Dolan \cite{bdolan}, proved in a
much strengthened form by Hopkins--Lurie and Lurie \cite{tft}, states
that a fully extended topological field theory (on framed manifolds)
is completely determined by its value for a point.
Analogously, but in a simpler way, a factorization algebra which is
systematically defined on all (framed) manifolds, is determined by
the $E_n$-algebra which appear as its local form \cite{francis}.

A sheaf is defined by its sections.
One is often more interested in the \emph{derived} sections, or the
cohomology.
Since we work in a homotopical setting for factorization algebras, the
sections we consider for an algebra are \emph{always} the `derived'
ones.
Thus, study of factorization algebra can be considered as study of a
kind of homology theory.
This homology theory, for locally constant algebras, was defined by
Lurie \cite{higher-alg}, and was called \emph{topological chiral
  homology}.
Following Francis and Costello (who works with not necessarily locally
constant algebras), we also call it \emph{factorization homology}.

\subsubsection{}
Let us now give some ideas for our main results.

\subsection{Descent properties of factorization algebras}
\setcounter{subsubsection}{-1}
\label{sec:descent-factorization-intro}

In the following, we assume that the target category $\cat{A}$ of
prealgebras is a symmetric monoidal infinity $1$-category which is
closed under sifted homotopy colimits, and that the monoidal
multiplication functors preserve sifted homotopy colimits
variable-wise.

We have developed descent properties of locally constant
factorization algebras for covers, and for bases of topology.
Our first result (\textbf{Theorem \ref{thm:factorizing-seifert-vk}},
note the conventions stated in Section \ref{sec:terminology-notation})
proves (as a particular case, see Example \ref{ex:factorizing-cover})
that topological chiral homology satisfies descent for a
\emph{factorizing cover} in the sense of Costello--Gwilliam \cite{cg}.
Therefore, this connects the `\emph{\v{C}ech}' approach of
Costello--Gwilliam to factorization homology, to Lurie's approach,
which is analogous to the \emph{singular} approach to the local
coefficient (co)homology.
(Costello--Gwilliam in fact considered not necessarily locally
constant algebras.) 
This, combined with ideas of Francis, lead to a proof of a version of
Francis' theorem \cite{francis}.
This will be contained in the sequel \cite{poincare} of this paper.
This theorem can be considered as giving an \emph{Eilenberg--Steenrod}
approach to factorization homology, and one concludes
from these theorems that all three approaches are equivalent.

Moreover, we have generalized Lurie's approach to factorization
homology in the following way.
Namely, his definition of topological chiral homology uses the basis
$\Disk(M)$ for the topology of a manifold $M$, consisting of open
subdisks.
He also uses disjoint unions of disks, which give another basis
$\Dis(M)$ of $M$.
This latter basis has a nice property in the spirit of
Costello--Gwilliam, which we might call here \kore{factorizingness}.
Lurie's definition is stated in terms of the pair
$\Disk(M)\to\Dis(M)$.

In \textbf{Theorem \ref{thm:general-notion-of-algebra}}, we have given
a sufficient condition for a pair $\cat{E}_1\to\cat{E}$ of bases to
define the same notion of a locally constant factorization algebra,
when it replaces the pair $\Disk(M)\to\Dis(M)$ in Lurie's definition.
Even though the theorem is slightly technical, the sufficient
condition we have found is easy to check in practice.
For example, it is quite easy to check whether we can find a suitable
$\cat{E}_1$ if $\cat{E}$ is a factorizing basis of $M$, closed
under disjoint union in $M$, and consists
of open submanifolds homeomorphic to disjoint unions of disks.

Thus, this theorem is useful, and in particular leads to the
following, as well as applications to be discussed in the next
section.
Let us denote by $\Alg_M(\cat{A})$, the infinity $1$-category of locally
constant factorization algebras on a manifold $M$.

\begin{theorem}[Theorem \ref{thm:descent-of-categories-of-algebras}]
\label{thm:descent-of-categories-of-algebras-intro}
The association $M\mapsto\Alg_M(\cat{A})$ (which is contravariantly
functorial in open embeddings) is a sheaf of infinity $1$-categories.
\end{theorem}

It follows that there is a notion of a locally constant factorization
algebra on an orbifold.

\subsection{Twisted product formula}
\setcounter{subsubsection}{-1}

As an application of our investigation of the descent properties of
factorization algebras, we have obtained the following basic theorem.
In the special case where the manifolds are the Euclidean spaces, we
recover a classical theorem of Dunn \cite{dunn}.
(See Remark \ref{rem:dunn-intro} below for the precise relation to his
theorem.)

\begin{theorem}[Theorem \ref{thm:prod-formula}]
\label{thm:prod-intro}
Let $B$, $F$ be manifolds.
Then, the restriction functor
\[
\Alg_{F\times B}(\cat{A})\longto\Alg_B(\Alg_F(\cat{A}))
\]
is an equivalence.
\end{theorem}

\begin{remark}
If one swaps the factors of $B\times F$, then on the side of algebras,
one recovers the canonical equivalence
$\Alg_B(\Alg_F)\equivwith\Alg_F(\Alg_B)$.
\end{remark}

\begin{remark}\label{rem:dunn-intro}
Dunn in fact obtains an equivalence at the level of operads
\cite{dunn}.
In particular, in his case, the equivalence of algebras holds
without any assumption on the target category.
Even though our theorem applies to any manifold, the equivalence in
this generality is proved only at the level of the category of
algebras in this paper, since our proof depends on the property of the
target category for the algebras.

Another slight difference with Dunn's result is that he considers
Boardman--Vogt's little cubes operad \cite{bv} instead of
factorization algebras on a Euclidean space.
We can use Theorem \ref{thm:general-notion-of-algebra} once again to
show that the difference is not essential.
See Remark \ref{rem:dunn} for the details.
\end{remark}

\begin{remark}\label{sec:product-formula-by-other}
A different proof of Theorem is obtained by Ginot by relying on Dunn's
theorem \cite{ginot}.
A version of Theorem for general (i.e., not assumed locally contant)
factorization algebras is described by Calaque in \cite{calaque} with
a (sketch of) proof by a strategy similar to ours (see Section
\ref{sec:related-work}).
We remark that the theorem for locally constant algebras may not be a
corollary of this since comparison of the ``locally constant'' objects
through Calaque's equivalence would perhaps not be straightforward.
\end{remark}
 
We have also obtained a natural generalization of this, where the
product is replaced by a fibre bundle (i.e., a `twisted' product).
In this case, the algebras on the right hand side needs to be twisted.
Namely, it should take values in an \emph{algebra} of categories on
$B$.
Once we allow this twisting, it is natural to consider further
twisting for algebras.
Namely, we consider algebras on the total space $E$ of a fibre bundle
taking values in a locally constant factorization algebra $\cat{A}$ of
categories on $E$.
For such $\cat{A}$, we have defined an algebra $\Alg_{E/B}(\cat{A})$
of categories on the base manifold $B$, which is a twisted version of
$\Alg_F$ in the previous theorem.
The following generalization of the previous theorem follows from (the
infinity $2$-categorical generalizations of) the previous theorem and
the descent results.

\begin{theorem}[Theorem \ref{thm:twisted}]
\label{thm:twisted-intro}
Let $B$ be a manifold, and let $E\to B$ be a smooth fibre bundle over
$B$.
For a locally constant factorization algebra $\cat{A}$ on $E$ of
infinity $1$-categories, there is a natural equivalence
\[
\Alg_E(\cat{A})\longequivto\Alg_B(\Alg_{E/B}(\cat{A}))
\]
of infinity $1$-categories, given by a suitable `restriction'
functor.
\end{theorem}

\begin{remark}
For this theorem, no assumption on sifted colimits are needed for
$\cat{A}$.
If $\cat{A}$ is instead a single fixed symmetric
monoidal category, there is actually a slight difference between an
algebra in $\cat{A}$ (for which Theorem \ref{thm:prod-intro} may fail
without assumption on sifted colimits), and an algebra taking values
in the `constant' algebra at $\cat{A}$ (to which
Theorem \ref{thm:twisted-intro} \emph{always} applies).
The assumption on sifted colimits simply ensures equivalence of these
two notions of an algebra.
\end{remark}

\subsection{Notes on related works}
\label{sec:related-work}
The descent property of the topological chiral homology for a
factorizing cover (as follows from Theorem
\ref{thm:factorizing-seifert-vk}, see Section
\ref{sec:descent-factorization-intro}) was proved earlier by
Ginot--Tradler--Zeinalian \cite{gtz}.
Their proof uses a theorem on the descent of the infinity $1$-category
of factorization algebras, similar to our Theorem
\ref{thm:descent-of-categories-of-algebras-intro} but in non-locally
constant setting.
The theorem is due to Costello--Gwilliam \cite{cg}.
Note that we prove Theorem
\ref{thm:descent-of-categories-of-algebras-intro} using Theorems
\ref{thm:general-notion-of-algebra} \emph{and}
\ref{thm:factorizing-seifert-vk}.
We do not know how to deduce Theorem
\ref{thm:descent-of-categories-of-algebras-intro} directly from the
theorem of Costello and Gwilliam.
The question is whether local constancy of a factorization algebra is
a `local' property in some useful manner, to which Theorem
\ref{thm:general-notion-of-algebra} gives one answer.

We learned about Calaque's work \cite{calaque} after our work was
completed.
He considers the notion of ``factorizing basis'' based on a similar
idea to our Definition \ref{def:factorizing-basis}.
Using this, he considers a theorem (Theorem 2.1.9 op.~cit.) which is
similar in spirit to our Theorem \ref{thm:general-notion-of-algebra},
for not necessarily locally constant factorization algebras.
Theorem \ref{thm:general-notion-of-algebra} is more involved than this
theorem, since it additionally answers a question on the
localness of local constancy as mentioned.
Calaque's proof of the product formula mentioned in Remark
\ref{sec:product-formula-by-other} uses this theorem ``2.1.9'',
similarly to our use of Theorem \ref{thm:general-notion-of-algebra}
for Theorem \ref{thm:prod-intro}.

\subsection{Notes on the relation to other articles by the author}
This paper, together with \cite{poincare} and the present author's
paper
\begin{itemize}
\item[{[a]}] \emph{Koszul duality between 
    $E_n$-algebras and coalgebras in a filtered category.}
  arXiv:1409.6943,
\end{itemize}
is based on his Ph.D.~thesis (accepted in April 2014).
The present article is logically independent of either of
\cite{poincare}, [a].

\subsection{Outline}
\setcounter{subsubsection}{-1}

\textbf{Section \ref{sec:terminology-notation}} is for introducing
conventions which are used throughout the main body.

In \textbf{Section \ref{sec:f-alg}}, we review Lurie's definitions and
results, and discuss descent properties of factorization algebras.

In \textbf{Section \ref{sec:generalize-apply}}, we discuss further
results including the twisted product formula.

\subsection*{Acknowledgment}
This paper is based on part of my Ph.D.~thesis.
I am particularly grateful to my advisor Kevin Costello for his
extremely patient guidance and continuous encouragement and support.
My contribution through this work to the subject of factorization
algebra can be understood as technical work of combining the ideas and
work of the pioneers such as Jacob Lurie, John Francis, and Kevin.
I am grateful to those people for their work, and for making
their ideas accessible.
Special thanks are due to John for detailed comments and
suggestions on the drafts of my thesis, which were essential for
many improvements of both the contents and exposition.
Many of those improvements were inherited by this paper.
I am also grateful to Owen Gwilliam for interesting conversations,
which directly influenced some parts of the present work.
I am grateful to Josh Shadlen, Owen, and Yuan Shen for their
continuous encouragement.

\section{Terminology and notations}
\label{sec:terminology-notation}
\setcounter{subsubsection}{-1}
\setcounter{equation}{-1}

\subsubsection{}

By a \kore{$1$-category}, we always mean an \emph{infinity}
$1$-category.
We often call a $1$-category (namely an infinity $1$-category) simply
a \kore{category}.
A category with discrete sets of morphisms (namely, a ``category''
in the more traditional sense) will be called a \emph{discrete}
category.

In fact, all categorical and algebraic terms will be used in
\emph{infinity} ($1$-) categorical sense without further notice.
Namely, categorical terms are used in the sense enriched in the
\emph{infinity} $1$-category of spaces, or equivalently, of infinity
groupoids, and algebraic terms are used freely in the sense
generalized in accordance with the enriched categorical structures.

For example, for an integer $n$, by an \emph{$n$-category}
(resp.~\emph{infinity} category), we mean an \emph{infinity}
$n$-category (resp.~infinity infinity category).
We also consider multicategories.
By default, multimaps in our multicategories will form
a \emph{space} with all higher homotopies allowed.
Namely, our ``\emph{multicategories}'' are ``infinity operads'' in the
terminology of Lurie's book \cite{higher-alg}.

\begin{remark}
We usually treat a space relatively to the structure of the standard
(infinity) $1$-category of spaces.
Namely, a ``\emph{space}'' for us is usually no more than an object of
this category.
Without loss of information, we shall freely identify a space in this
sense with its fundamental infinity groupoid, and call it also a
``\emph{groupoid}''.
Exceptions in which the term ``space'' means not necessarily
this, include a ``Euclidean space'', the ``total space'' of a fibre
bundle, etc., in accordance with the common customs.
\end{remark}

\subsubsection{}
If $\cat{C}$ is a category and $x$ is an object of $\cat{C}$,
then we denote by $\cat{C}_{/x}$, the ``\emph{over}'' category, of
objects of
$\cat{C}$ lying over $x$, i.e., equipped with a map to $x$.
We denote the ``\emph{under}'' category for $x$, in other words,
$\bigl((\cat{C}^\op)_{/x}\bigr)^\op$, by $\cat{C}_{x/}$.

More generally, if a category $\cat{D}$ is equipped with a functor to
$\cat{C}$, then we define
$\cat{D}_{/x}:=\cat{D}\times_{\cat{C}}\cat{C}_{/x}$, and similarly for
$\cat{D}_{x/}$.
Note here that $\cat{C}_{/x}$ is mapping to $\cat{C}$ by the functor
which forgets the structure map to $x$.
Note that the notation is abusive in that the name of the functor
$\cat{D}\to\cat{C}$ is dropped from it.
In order to avoid this abuse from causing any confusion, we shall use
this notation only when the functor $\cat{D}\to\cat{C}$ that we are
considering is clear from the context.

\subsubsection{}

By the \kore{lax colimit} of a diagram in the category $\Cat$ of
categories (of a limited size), indexed by a category $\cat{C}$, we
mean the Grothendieck construction.
We choose the variance of the laxness so the lax colimit projects to
$\cat{C}$, to make it an op-fibration over $\cat{C}$, rather
than a fibration over $\cat{C}^\op$.
(In particular, if $\cat{C}=\cat{D}^\op$, so the functor is
contravariant on $\cat{D}$, then the familiar fibred category over
$\cat{D}$ is the \emph{op}-lax colimit over $\cat{C}$ for us.)
Of course, we can choose the variance for lax \emph{limits} compatibly
with this, so our lax colimit generalizes to that in any $2$-category.

\section{Descent properties of factorization algebras}
\label{sec:f-alg}
\setcounter{subsection}{-1}
\setcounter{equation}{-1}

\subsection{Introduction}

In this section, we introduce the notion of a locally constant
factorization algebra following Lurie (although he did not use this
particular term), and then investigate its descent properties.
This will be a study of the descent properties of Lurie's
``topological chiral homology''.

Many notions and notations we introduce in this section are from
Lurie's book ``Higher Algebra'' \cite{higher-alg}, which has an index
and an index for notations.

\subsection{Locally constant factorization algebra}
\setcounter{subsubsection}{-1}

Given a manifold $M$, let us denote by $\Open(M)$ the poset of open
submanifolds of $M$.
It (considered as a category where a map is an inclusion) has a
partially defined symmetric monoidal structure given by the
\emph{disjoint} union in $M$,
$\Disj_S\colon\Open(M)^{(S)}\to\Open(M)$, where the domain here is
the full subposet of $\Open(M)^S$ consisting of \emph{pairwise
  disjoint} family of open submanifolds of $M$ indexed by the finite
set $S$.

\begin{definition}
Let $\cat{A}$ be a symmetric monoidal category.
Then a \kore{prefactorization algebra} (or just a
``\kore{prealgebra}'') on $M$ (valued) in $\cat{A}$, is a symmetric
monoidal functor $\Open(M)\to\cat{A}$.

We say that a prealgebra is \kore{locally constant} if $A$ takes every
inclusion $D\into D'$ between disks in $M$ (namely, open submanifolds
which is homeomorphic to an open disk), to an equivalence
$A(D)\equivto A(D')$.

The category of \emph{locally constant} prealgebras on $M$ in
$\cat{A}$ will be denoted by $\Pre\Alg_M(\cat{A})$.
\end{definition}

Let $M$ be a manifold.
Let $n$ denote its dimension.
Then, following Lurie, we denote by $\Disk(M)$, the poset consisting
of open submanifolds $U\sub M$
homeomorphic to an open disk
of dimension $n$ (by an unspecified homeomorphism).
This poset has a structure of a symmetric multicategory where a
multimap is a disjoint inclusion in $M$, so for every fixed source and
target, the space of multimaps is either empty or contractible.

Recall that given symmetric multicategories $\cat{A}$, $\cat{B}$, an
\kore{algebra} on $\cat{B}$ in $\cat{A}$ is a morphism
$\cat{B}\to\cat{A}$ of symmetric multicategories.

The following is a notion equivalent to an algebra over Lurie's
multicategory $\Elu_M$ from \cite{higher-alg}.
See Theorem 5.2.4.9 there, also restated here as Theorem
\ref{thm:lurie-multilocalization}.
Another equivalent notion has a natural name, and we use that
name.
All notions and equivalence between them will be reviewed below.

\begin{definition}
Let $\cat{A}$ be a symmetric monoidal category.
Then a \kore{locally constant factorization algebra} (or just a
``(locally constant) \kore{algebra}'', often in this work) on $M$
valued in $\cat{A}$, is an algebra on $\Disk(M)$ in $\cat{A}$ whose
underlying functor (of ``colours'') inverts any map in $\Disk(M)$
(which is an inclusion of a single disk into another).
The category of locally constant algebras on $M$ in $\cat{A}$ will be
denoted by $\Alg_M(\cat{A})$.
\end{definition}

\begin{remark}
This definition makes sense for $\cat{A}$ just a
symmetric multicategory, but for comparison with other notions, it is
convenient to have $\cat{A}$ to be symmetric monoidal.
\end{remark}

Following Lurie, let us denote by $\Dis(M)$ the poset of open
submanifolds $U\sub M$
homeomorphic (by an unspecified homeomorphism) to the disjoint union
of a finite number of disks.
It has a partially defined monoidal structure given by the disjoint
union in $M$.
There is a functor $\Disk(M)\to\Dis(M)$ of multicategories, so a
symmetric monoidal functor $A\colon\Dis(M)\to\cat{A}$ to a symmetric
monoidal category $\cat{A}$ restricts to a
morphism $\Disk(M)\to\cat{A}$ of symmetric multicategories.
Moreover, any morphism $\Disk(M)\to\cat{A}$ with $\cat{A}$ symmetric
monoidal category extends
uniquely to a symmetric monoidal functor $\Dis(M)\to\cat{A}$.
Namely, an algebra on $M$ can be also described as a symmetric
monoidal functor $\Dis(M)\to\cat{A}$.

\begin{remark}
Again, this is still true if the monoidal structure of $\cat{A}$ is
only partially defined, but this is not an important point for us.
\end{remark}

Note that there is a (necessarily symmetric) monoidal embedding
$\Dis(M)\into\Open(M)$.
Given a functor $\Dis(M)\to\cat{A}$, one has its left Kan extension
$\Open(M)\to\cat{A}$ at least if $\cat{A}$ has colimits.

If the monoidal multiplication in $\cat{A}$ distributes over colimits, then
the Kan extension $\Open(M)\to\cat{A}$ of a symmetric monoidal functor
$\Dis(M)\to\cat{A}$ becomes symmetric monoidal in a unique way, so its
restriction to $\Dis(M)$ becomes the original symmetric monoidal
functor.
In fact, Lurie proves that relevant colimits here can be described as
sifted colimits (see Propositions \ref{prop:cofinality-of-inversion}
and \ref{prop:lurie-siftedness} below).
Therefore, it sufficed to consider just sifted colimits.

To summarize, if the target category $\cat{A}$ has sifted colimits,
and the monoidal multiplication in $\cat{A}$ distributes over sifted colimits
(equivalently, sifted colimits are preserved by the monoidal multiplication),
then we have a functor $\Alg_M(\cat{A})\to\Pre\Alg_M(\cat{A})$ given
by left Kan extension.
This functor is clearly fully faithful, and it is left adjoint to the
functor given by restriction through the functor $\Disk(M)\to\Open(M)$
of symmetric multicategories.
In this way, $\Alg_M(\cat{A})$ is a right localization of
the category of locally constant prealgebras.

Within the category of locally constant prealgebras, the algebras can
be characterized as those prealgebras which, as a functor, is the
left Kan extension of its restriction to $\Dis(M)$.
We often identify $\Alg_M$ with this right localized full subcategory
of $\Pre\Alg_M$.

The following is basic.
\begin{example}[See also Francis'
  \cite{francis}]\label{ex:sheaf-as-coalgebra}
Let $\cat{A}$ be a category closed under small colimits, and let us
consider it as a symmetric monoidal category under the Cartesian
coproduct.
This symmetric monoidal multiplication
$\cat{A}\times\cat{A}\to\cat{A}$ takes colimits in
$\cat{A}\times\cat{A}$ to colimits in the target, so sifted colimits
are preserved variablewise, so the arguments above applies to this
symmetric monoidal structure.

In this case, any functor $\Dis(M)\to\cat{A}$ has a unique lax
symmetric monoidal structure, and this structure is strong monoidal if
and only if the functor is the left Kan extension (in the canonical
way) from its restriction to $\Disk(M)$.

It follows that a locally constant algebra in $\cat{A}$ with respect
to the Cartesian coproduct, is the same thing as a locally constant
cosheaf in $\cat{A}$.

Dually, if $\cat{A}$ is closed under limits, then locally constant
algebra in $\cat{A}^\op$ with respect to the Cartesian product of
$\cat{A}$, is the same thing as a locally constant \emph{sheaf} valued
in $\cat{A}$.
\end{example}

\subsection{Assumption on the target category}
\label{sec:assumption-on-colimit}
\setcounter{subsubsection}{-1}

From now on, in this paper, we assume that the target category
$\cat{A}$ of prealgebras has sifted colimits, and the monoidal multiplication
functor on $\cat{A}$ preserves sifted colimits variable-wise.
Equivalently, the monoidal multiplication should preserve sifted
colimits for all the variables at the same time.

\subsection{Descent for factorizing covers}
\setcounter{subsubsection}{-1}

Note our assumption just stated.

For a prealgebra on $M$, being the Kan extension of its restriction to
disjoint union of disks is a kind of descent property.
We shall observe a more general descent satisfied by a locally
constant algebra.

\begin{definition}
Let $\cat{C}$ be a category and let
$\chi\colon\cat{C}\to\Open(M)$ be a functor.
For $i\in\cat{C}$, denote $\chi(i)$ also by $U_i$ within this
definition.
We shall call this data a \kore{factorizing cover} which is \kore{nice
  in Lurie's sense}, or briefly, \kore{factorizing l-nice cover}, of
$M$ if for any non-empty finite subset $x\sub M$, the full subcategory
$\cat{C}_x:=\{i\in\cat{C}\:|\:x\subset U_i\}$ of $\cat{C}$ has
contractible classifying space.
\end{definition}

\begin{remark}\label{rem:def-factorizing-cover}
The definition is inspired by the definition of a factorizing cover by
Costello--Gwilliam \cite{cg}, and a condition introduced by Lurie for
his \emph{generalized Seifert--van Kampen theorem}
\cite[Appendix]{higher-alg}.
``Nice'' is Lurie's description of a cover satisfying his conditions,
where he does not intend this to be a part of his terminology.
However, we borrow this word ``nice'' and make it our term for the
notion above, for unfortunate lack of creativity for a better name.
\end{remark}

\begin{example}
If $M$ is empty, then any cover of $M$, including the one indexed by
the empty category, is factorizing l-nice.
\end{example}

\begin{example}
The inclusion $\Dis(M)\to\Open(M)$ determines a factorizing l-nice
cover.
\end{example}

\begin{example}
Consider a cover of $M$ by a filtered (or ``directed'') inductive
system of open submanifolds of $M$.
Then this cover is factorizing l-nice.
\end{example}

\begin{example}\label{ex:factorizing-cover}
Suppose given an open cover $\cover{U}=\{U_s\}_{s\in S}$ of $M$
indexed by a set $S$.
For simplicity, assume that this cover is closed under taking finite
disjoint union.
If this is not satisfied, replace $S$ by the set of finite subsets $T$
of $S$ for which $U_t$ are pairwise disjoint for $t\in T$.
(For example, if $M=\kara$, then we are excluding the empty cover
indexed by $S=\kara$.)

Denote by $\Simp_{/S}$ the category of combinatorial simplices whose
vertices are labeled by elements of $S$.
Namely, its objects are finite non-empty ordinal $I$ equipped with a
set map $s\colon I\to S$.
Then the cover determines a functor
$\chi\colon(\Simp_{/S})^\op\to\Open(M)$ by
\[
(I,\,s\colon I\to S)\longmapsto U_s:=\Intersect_{i\in I}U_{s(i)}.
\]

In Costello--Gwilliam's terminology, the cover $\cover{U}$
is \kore{factorizing} if for this $\chi$, the category
$(\Simp_{/S})^\op_x$ is non-empty for every finite subset $x\sub M$
(equivalently if there is $i\in S$ for which $x\sub U_i$).

It is immediate to see that $\chi$ determines a factorizing l-nice
cover if (and only if) the cover is factorizing in
Costello--Gwilliam's sense.

Given a prealgebra $A$ on $M$, the descent complex for $\cover{U}$ of
Costello--Gwilliam is equivalent to $\colim_{(\Simp_{/S})^\op}A$.
\end{example}

The following generalizes the Kan extension property from the values
for disjoint union of disks.
\begin{theorem}\label{thm:factorizing-seifert-vk}
Let $A$ be a locally constant algebra on $M$ (in a symmetric monoidal
category $\cat{A}$ satisfying our conditions stated in
Section \ref{sec:assumption-on-colimit}).
Then for any factorizing l-nice cover determined by $\chi\colon
\cat{C}\to\Open(M)$, the map
$A(M)\from\colim_\cat{C}A\chi$ is an equivalence.
\end{theorem}

For the proof, we need another description of locally constant
algebras, due to Lurie.
We shall give the proof after we give the description in
Section \ref{sec:isotopy-invariance}.

\subsection{Isotopy invariance}
\label{sec:isotopy-invariance}
\setcounter{subsubsection}{-1}

\subsubsection{}

Let $M$ be a manifold, and let $n$ be its dimension.

Let $\Elu_M$ be the multicategory (i.e., an ``infinity operad'')
introduced by Lurie.
Its objects are the open submanifolds of $M$ homeomorphic to a disk
of dimension $n$.
The space of multimaps $\{U_i\}_{i\in S}\to V$ is that formed by an
embedding $f\colon\coprod_iU_i\into V$ together with an isotopy on
each $U_i$ from the defining inclusion $U_i\into M$ to $f\colon
U_i\into M$.

It is immediate from this description that the underlying category
(the category of ``colours'') of
$\Elu_M$ is a groupoid equivalent to (the fundamental infinity
groupoid of) the space naturally formed by its objects.

Consider the obvious morphism $\Disk(M)\to\Elu_M$ of multicategories.

\begin{theorem}[Lurie, Theorem 5.2.4.9 of \cite{higher-alg}]
\label{thm:lurie-multilocalization}
Restriction through the morphism $\Disk(M)\to\Elu_M$ induces a fully
faithful functor between the categories of algebras
on these multicategories.
The essential image of the functor consists precisely of the locally
constant algebras on $M$.
\end{theorem}
In particular, a locally constant algebra on $M$ extends uniquely (up
to a contractible space) to an algebra on $\Elu_M$.

The property of an algebra on disks that it extends to $\Elu_M$, can
be understood as isotopy invariance (where the way to be invariant can
be specified functorially) of the functor.
By the above theorem, this property is equivalent to being locally
constant.

\subsubsection{}
Let $\ccat{D}(M)$ be as defined by Lurie (Definition 5.3.2.11 of
\cite{higher-alg}).
Its objects are open submanifolds of $M$ which are homeomorphic to a
finite disjoint union of disks.
The space of maps $U\to V$ is the space formed by embeddings $f\colon
U\into V$ together with an isotopy from the defining
inclusion $U\into M$ to $f\colon U\into M$.

Disjoint union in $M$ cannot be made into a partial monoidal structure
on $\ccat{D}(M)$ since the isotopies we used in defining a morphism in
$\ccat{D}(M)$, was required to be isotopies on the whole $U$, not just
on each of its components.
However, $\ccat{D}(M)$ can be extended to a symmetric partial monoidal
category which has the same objects but where the mentioned
restriction on the maps is discarded.
Let us denote this partial monoidal category by $\parmo\Elu_M$.
The composite $\Elu_M\to\ccat{D}(M)\to\parmo\Elu_M$ then has a
canonical structure as a map of multicategories, and we can try to
extend $A$ to a symmetric monoidal functor on $\parmo\Elu_M$.

To see that this is possible, let us further try discarding the
restriction on the \emph{objects}.
Namely, an object of $\parmo\Elu_M$ is an object of $\ccat{D}(M)$,
which can be considered as a disjoint family of disks in $M$, but we
can instead include \emph{any} family of disks (and define morphisms
in the same way as in $\parmo\Elu_M$).
The result is the symmetric monoidal category freely generated from
$\Elu_M$.
Therefore, an algebra $A$ on $\Elu_M$ can be extended to a symmetric
monoidal functor on the free symmetric monoidal category, and then be
restricted to $\parmo\Elu_M$ through the symmetric monoidal inclusion.
This symmetric monoidal functor on $\parmo\Elu_M$, as an algebra on a
multicategory, extends the algebra $A$ on $\Elu_M$.

Moreover, there is a commutative square
\[\begin{CD}
\Disk(M)@>>>\Elu_M\\
@VVV@VVV\\
\Dis(M)@>>>\ccat{D}(M)
\end{CD}\]
which, with the functor $\ccat{D}(M)\to\parmo\Elu_M$, factorizes a
square
\[\begin{CD}
\Disk(M)@>>>\Elu_M\\
@VVV@VVV\\
\Dis(M)@>>>\parmo\Elu_M.
\end{CD}\]
where the bottom functor underlies a symmetric monoidal functor.
It follows that, by restricting to $\ccat{D}(M)$ (the underlying
functor of) the described symmetric monoidal functor on $\parmo\Elu_M$
extending $A$, one gets a functor on $\ccat{D}(M)$ which extends both
(the underlying functor of) $A$ on $\Elu_M$, and (the underlying
functor of) the symmetric monoidal functor on $\Dis(M)$ uniquely
extended from the algebra $A\resto{\Disk(M)}$ on $\Disk(M)$.

\begin{proposition}[Lurie, Proposition 5.3.2.13 (1) of \cite{higher-alg}]
\label{prop:cofinality-of-inversion}
The functor $\Dis(M)\to\ccat{D}(M)$ is cofinal.
\end{proposition}

That is, for a functor defined on $\ccat{D}(M)$, its colimit over
$\ccat{D}(M)$ gives the colimit of the restriction of the same functor
to $\Dis(M)$.

\begin{proposition}[Lurie, Proposition 5.3.2.15 of \cite{higher-alg}]
\label{prop:lurie-siftedness}
The category $\ccat{D}(M)$ is sifted.
\end{proposition}

\begin{corollary}\label{cor:isotopy-invariance}
Let $A$ be a locally constant algebra on $M$.
Consider it as an algebra on $\Elu_M$, and then extend
its underlying functor to $\ccat{D}(M)$ in the explained way.
Denote the resulting functor on $\ccat{D}(M)$ still by $A$.
Then the canonical map
\[
A(M)\longfrom\colim_{\ccat{D}(M)}A
\]
is an equivalence.
\end{corollary}

\subsubsection{}
We can now start a proof of Theorem \ref{thm:factorizing-seifert-vk}.
Recall that a functor $\cat{C}\to\cat{D}$ is \emph{cofinal} if for
every functor $f$ with domain $\cat{D}$, $\colim f$ (when this
exists) is a colimit of $f$ over $\cat{C}$ (in the canonical way).
(Lurie \cite{topos} Definition 4.1.1.1, but see also Proposition
4.1.1.8.)

\begin{definition}
Let $\cover{U}$ be a cover of a manifold $M$, given by a functor
$\chi\colon\cat{C}\to\Open(M)$, $i\mapsto U_i$.
Then $\cover{U}$ is said to be \kore{effectively factorizing l-nice}
if the canonical functor
\[
\colim_i\ccat{D}(U_i)\longto\ccat{D}(M)
\]
is cofinal.
\end{definition}

\begin{remark}\label{rem:van-kampen-cofinality}
By Proposition \ref{prop:cofinality-of-inversion}, the condition of
being an effectively factorizing l-nice cover is equivalent to that
the functor
\[
\colim_i\Dis(U_i)\longto\ccat{D}(M)
\]
is cofinal.
\end{remark}

Theorem \ref{thm:lurie-multilocalization} immediately implies the
following.

\begin{lemma}
\label{lem:factorizing-seifert-vk}
Let $A$ be a locally constant algebra on $M$.
Then for any effectively factorizing l-nice cover determined by
$\chi\colon\cat{C}\to\Open(M)$, the canonical map
$A(M)\from\colim_\cat{C}A\chi$ is an equivalence.
\end{lemma}

Theorem \ref{thm:factorizing-seifert-vk} is an immediate
consequence of this and the
following, `factorizing' version of Lurie's higher homotopical
generalization of the Seifert--van Kampen theorem.
The factorizing version is actually a consequence of the original
theorem.
Our proof will be similar to the proof of Theorem 5.1 of the paper
\cite{bw} by Boavida de Brito--Weiss, and will also use some arguments
similar to those from the proofs of the theorems above of Lurie.

\begin{proposition}\label{prop:van-kampen-cofinality}
Let $M$ be a manifold.
Then every factorizing l-nice cover of $M$ is effectively factorizing
l-nice.
\end{proposition}

In the proof, we shall use the following standard fact from basic
homotopy theory.
Its proof is included for completeness.

\begin{lemma}\label{lem:standard-cofinality}
Let $\cat{G}$ be a groupoid.
Then a functor $\cat{C}\to\cat{G}$ from a $1$-category is cofinal if
(and only if) the induced map $B\cat{C}\to\cat{G}$ is an equivalence.
\end{lemma}
\begin{proof}\setcounter{proofsec}{0}
\proofsec
Assuming that $\cat{G}=B\cat{C}$, we want to prove that the colimit of
any functor $L$ defined over $\cat{G}$ is a colimit of $L$ over
$\cat{C}$.
(``Only if'' part is trivial since $B\cat{C}$ is a colimit of the
final diagram over $\cat{C}$ in the $1$-category of groupoids.)

Note that it suffices to consider the case where $L$ is taking values
in the opposite of the category of spaces, since whether an object is
a colimit is tested by homming to another object.
Let us conveniently change the variance of $\cat{C}$ and $\cat{G}$,
and consider the limits of a \emph{covariant} functor $L$ defined on
$\cat{G}$.
Thus, we want to prove that for $\cat{G}=B\cat{C}=\colim_\cat{C}*$,
colimit taken in the category of groupoids, the induced map
$\lim_\cat{G}L\to\lim_\cat{C}L$ is an equivalence.

The crucial fact here is that for any object $i$ of $\cat{G}$, $L(i)$
is the homotopy fibre of the projection $\colim_\cat{G}L\to\cat{G}$.
Namely, $L(i)$ is the space of sections of this map over the point
$i$.

It follows that $\lim_\cat{C}L$ is the space of global sections if
$\cat{G}=\colim_\cat{C}*$.
Thus, we have proved that $\lim_\cat{C}L$ is functorially equivalent
to a space which is independent of $\cat{C}$ as long as the map
$B\cat{C}\to\cat{G}$ is an equivalence.
(In particular, this independent space is identified with
$\lim_\cat{G}L$ through the equivalence obtained in the case
where the functor $\cat{C}\to\cat{G}$ is an equivalence.)
This completes the proof.

\proofsec
Alternatively, one can apply Joyal's generalization of Quillen's
Theorem A \cite{topos}, although as we have shown, this is not
necessary.
Again, assuming $\cat{G}=B\cat{C}$, we want to show that, for any
object $x$ of $\cat{G}$, the under category $\cat{C}_{x/}$ has
contractible classifying space.

The point is that, since $\cat{G}$ is a groupoid, $\cat{C}_{x/}$
coincides with the fibre of the functor $\cat{C}\to\cat{G}$ over $x$.
The result is immediate from this since the geometric realization
functor preserves pull-backs.
\end{proof}

\begin{proof}[Proof of Proposition \ref{prop:van-kampen-cofinality}]
Suppose that a factorizing l-nice cover $\cover{U}$ of $M$ is given by
a functor $\chi\colon\cat{C}\to\Open(M)$, $i\mapsto U_i$.
We want to show that the functor
\[
\colim_i\ccat{D}(U_i)\longto\ccat{D}(M)
\]
is cofinal.

Recall that for open $U\sub M$, the category $\ccat{D}(U)$ was a comma
category in the category $\Man$ of category of manifolds, in which the
space of morphisms is the space of open embeddings.
Namely, let $\ccat{D}$ be the full subcategory of $\Man$ of manifolds
whose objects are equivalent to disjoint union of disks of dimension
$n$, where $n=\dim M$.
Then $\ccat{D}(U)$ was the comma category whose object was a
morphism from an object of $\ccat{D}$ to $U$.

In other words, $\ccat{D}(U)=\laxcol_{D\in\ccat{D}}\Emb(D,U)$, where
$\Emb(D,U)$, the infinity groupoid of embeddings, is the space of
morphisms in $\Man$, and the lax colimit is taken in the $2$-category
of categories.

It follows that it suffices to prove for every
$D\in\ccat{D}$, the map
\begin{equation*}
\laxcol_{i\in\cat{C}}\Emb(D,U_i)\longto\Emb(D,M)
\end{equation*}
is cofinal for every $D\in\ccat{D}$.

In view of Lemma \ref{lem:standard-cofinality} above, it suffices to
prove that the map
\[
\colim_{i\in\cat{C}}\Emb(D,U_i)\longto\Emb(D,M)
\]
is an equivalence.

Choose a homeomorphism $D\equivwith S\times\R^n$ for a finite set $S$.
In particular, we have picked a point in each component of $D$,
corresponding to the origin in $\R^n$, together with a germ of chart
at the chosen points.
Then, given an embedding $D\into U$, restriction of it to the
germs of charts at the chosen points gives us an injection $S\into U$
together with germs of charts in $U$ at the image of $S$.
This defines a homotopy equivalence of $\Emb(D,U)$ with the space of
germs of charts around distinct points in $U$, labeled by $S$.

Furthermore, for any $U$, this space is fibred over the configuration
space $\Conf(S,U):=\Emb(S,U)/\Aut(S)$, with fibres equivalent to
$\mathrm{Germ}_0(\R^n)\wreath\Aut(S)$, where $\mathrm{Germ}_0(\R^n)$
is from \cite[Notation 5.2.1.9]{higher-alg}.

Thus it suffices to show that the map
\[
\colim_{i\in\cat{C}}\Conf(S,U_i)\longto\Conf(S,M)
\]
is an equivalence of spaces.

In order to prove this, Lurie's generalized Seifert--van Kampen
theorem implies that it suffices to prove that for every
$x\in\Conf(S,M)$, the category $\{i\in\cat{C}\:|\:x\in\Conf(S,U_i)\}$
has contractible classifying space.
However, $x\in\Conf(S,U_i)$ is equivalent to $\supp x\sub U_i$, where
$\supp x$ is the subset of $M$ corresponding to the configuration $x$,
so the required condition is exactly our assumption that the cover is
factorizing l-nice.
\end{proof}

\subsection{Basic descent}
\setcounter{subsubsection}{-1}

\subsubsection{}

We continue with the assumptions introduced in
Section \ref{sec:assumption-on-colimit}.
Namely, we assume that the target category
$\cat{A}$ of prealgebras has sifted colimits, and the monoidal multiplication
functors on $\cat{A}$ preserve sifted colimits (variable-wise).

\begin{definition}\label{def:factorizing-basis}
Let $M$ be a manifold, and let $\cover{U}$ be an effectively
factorizing l-nice cover of $M$, given by a functor
$\chi\colon\cat{C}\to\Open(M)$, $i\mapsto U_i$.
We say that $\cover{U}$ is an \kore{(effectively) factorizing l-nice
  basis} for the topology of $M$, if for every open $V\sub M$, the
functor $\chi\colon\cat{C}_{/V}\to\Open(M)_{/V}=\Open(V)$ determines
an (effectively) factorizing l-nice cover of $V$.
\end{definition}

\begin{remark}
There is an obvious non-factorizing version of these notions.
\end{remark}

It is immediate to see that a factorizing l-nice basis is effectively
so as well.

\begin{example}
Disjoint open disks of $M$ form a factorizing l-nice basis of $M$.
\end{example}

We have the following as a corollary of Lemma
\ref{lem:factorizing-seifert-vk}, in view of the definition of an
effectively factorizing l-nice basis.

\begin{proposition}\label{prop:kan-extension-from-basis}
Let $M$ be a manifold with an effectively factorizing l-nice basis
$\cover{U}$.
Then any factorization algebra $A$, as a functor, is a left Kan
extension of its restriction to $\cover{U}$, namely, if the basis is
given by a functor $\chi\colon\cat{C}\to\Open(M)$, then $A$ is a Kan
extension along $\chi$ of $A\chi$.
\end{proposition}

In fact, the converse to this is true in the following sense.

\begin{proposition}\label{prop:descent-of-algebras}
Let $M$ be a manifold with an effectively factorizing l-nice basis
$\cover{U}$.
Suppose $A$ is a prealgebra on $M$, then it is a locally
constant factorization algebra if (and only if) it satisfies the
following.
\begin{itemize}
\item For any basic (in the basis) open $U$, the conditions
\begin{enumerate}\setcounter{enumi}{-1}
\item\label{it:local-constancy} $A$ is locally constant when
  restricted to $U$,
\item\label{it:right-value-on-basis}
  the map $\colim_{\Dis(U)}A\to A(U)$ is an equivalence
\end{enumerate}
are satisfied, and
\begin{enumerate}\setcounter{enumi}{1}
\item\label{it:kan-ext-from-basis}
  the underlying functor of $A$ is a left Kan extension of its
  restriction to the basis.
\end{enumerate}
\end{itemize}
\end{proposition}

\begin{theorem}\label{thm:basic-descent-of-category-of-algebras}
The association $M\mapsto\Alg_M(\cat{A})$ (which is contravariantly
functorial in open embeddings of codimension $0$) satisfies
descent for any effectively factorizing l-nice basis.
\end{theorem}
\begin{proof}[Proof assuming Proposition]
If $A$ is a locally constant factorization algebra on a manifold $U$,
then the conditions \eqref{it:local-constancy} and
\eqref{it:right-value-on-basis} of Proposition are satisfied.
\end{proof}

Let us seek for a proof of Proposition.
Having Proposition \ref{prop:kan-extension-from-basis}, the only
non-trivial point of proof would be in showing that $A$ is locally
constant.
Although Proposition \ref{prop:descent-of-algebras} can be proved in a
direct manner, we shall deduce it from a similar theorem in a more
specific situation, with weaker looking local constancy assumption.
The weaker assumption is more flexible, and the theorem will turn out
to be useful.

The theorem is as follows.
(We shall use its corollary \ref{cor:descent-of-general-algebra} for
our proof of Proposition \ref{prop:descent-of-algebras}.)

\begin{theorem}\label{thm:general-notion-of-algebra}
Let $M$ be a manifold, and let $\cover{V}$ be an effectively
factorizing l-nice basis of $M$, given by a (necessarily symmetric)
monoidal functor $\psi\colon\cat{E}\to\Open(M)$, $i\mapsto V_i$, from
a symmetric partial monoidal category $\cat{E}$, landing in fact in
$\Dis(M)$.
Let $\cat{E}_1$ be a category mapping to (the underlying category of)
$\cat{E}$, for which the hypotheses \ref{hy:monoidal-basis}
below are satisfied.
Then a prealgebra $A$ in $\cat{A}$ on $M$ is a locally constant
factorization algebra on $M$ if and only if it satisfies the
following.
\begin{enumerate}
\setcounter{enumi}{-1}
\item $A\psi$ sends every morphism in $\cat{E}_1$ to an equivalence.
\item The underlying functor of $A$ is a left Kan extension of its
      restriction $A\psi$ to the factorizing basis.
\end{enumerate}
\end{theorem}

In other words, any pair $\cat{E}_1\to\cat{E}$
satisfying the hypotheses can replace the pair
$\Disk(M)\to\Dis(M)$ in the definition of a locally constant
factorization algebra.

\begin{remark}\label{rem:niceness-of-basis-by-disks}
For every $U\sub M$, the section
$\cat{E}_{/U}\to\laxcol_{i\in\cat{E}_{/U}}\Dis(V_i)$ to the canonical
functor
$\laxcol_{i\in\cat{E}_{/U}}\Dis(V_i)\to\laxcol_{i\in\cat{E}_{/U}}*=\cat{E}_{/U}$,
sending $i$ to the image of the (existing!) terminal object of
$\Dis(V_i)$ in the colimit, is cofinal.

In particular, the assumption that the basis is effectively
factorizing l-nice is equivalent to that the composite
\[
\cat{E}_{/U}\xrightarrow{\psi}\Dis(U)\longto\ccat{D}(U)
\]
is cofinal for every $U$, since this can be written as the composite
\[
\cat{E}_{/U}\longto\colim_{i\in\cat{E}_{/U}}\Dis(V_i)\longto\ccat{D}(U).
\]
See Remark \ref{rem:van-kampen-cofinality}.
\end{remark}

\subsubsection{}
We need to introduce some notation to state the hypotheses.
Note that, a map $f\colon D\to E$ in $\ccat{D}(M)$ is an equivalence
if and only if the embedding $D\into E$ (call it $g$) contained as a
part of data determining $f$, is the disjoint union of embeddings of a
single disk into another.
That is, if and only if there is a one-to-one correspondence between
the connected components of $D$ and those of $E$, such that $g$ embeds
each component of $D$ into the corresponding component of $E$.

Given a finite set $S$, we denote by $\ccat{D}_S(M)$ the
groupoid whose objects are families $D=(D_s)_{s\in S}$ of disks
labeled by elements of $S$, and
pairwise disjointly embedded in $M$, and a morphism $D\to
E=(E_s)_{s\in S}$ is an equivalence $\Disj_SD\equivto\Disj_SE$ in
$\ccat{D}(M)$ which preserves the labels.
(Note the difference of this with just maps $D_s\to E_s$ for every
$s\in S$.)

Analogously, let $\Dis_S(M)$ denote the poset whose objects are
families $D=(D_s)_{s\in S}$ of disks labeled by elements of $S$, and
pairwise disjointly embedded in $M$, and a morphism $D\to
E=(E_s)_{s\in S}$ is an inclusion in $M$ such that $D_s\sub E_s$ for
every $s\in S$.
(This is the same as a family of inclusions labeled by $s\in S$.)
For example, if $S$ consists of $1$ element, then $\Dis_S=\Disk$.

\begin{lemma}\label{lem:cofinality-with-label}
The functor $\Dis_S(M)\to\ccat{D}_S(M)$ is cofinal.
\end{lemma}
\begin{proof}
By Lemma \ref{lem:standard-cofinality}, it suffices to prove that this
functor identifies the groupoid $\ccat{D}_S(M)$ with the classifying
space of $\Dis_S(M)$.
The argument for this will be similar to the proof of Proposition
\ref{prop:van-kampen-cofinality}.

Namely, let $D=(D_s)_{s\in S}$ be a family of free (i.e., not
embedded) disks indexed by the elements of $S$.
Then we deduce as before that it suffices to prove that the map
\[
\colim_{U\in\Dis_S(M)}\prod_{s\in
  S}\Emb\bigl(D_s,U_s\bigr)\longto\Emb\bigl(\textstyle{\coprod_SD,M}\bigr)
\]
is an equivalence.

It further follows in the similar manner as before, that it suffices
to prove that the map
\[
\colim_{U\in\Dis_S(M)}\prod_SU\longto\Conf(S,M)
\]
is an equivalence.

The equivalence follows from applying the generalized
Seifert--van Kampen theorem to the following open cover of
$\Conf(S,M)$.
The cover is indexed by the category $\Dis_S(M)$, and is given by the
functor which associates to $U\in\Dis_S(M)$, the open subset
$\prod_SU$ of $\Conf(S,M)$.
It is immediate to see that this cover satisfies the assumption for
the generalized Seifert--van Kampen theorem.
\end{proof}

\begin{remark}
Note that the last step of Proof of Lemma
\ref{lem:cofinality-with-label} implies that the classifying space of
$\Dis_S(M)$ is equvalent to the labeled configuration space of $M$.
Namely, the groupoid $\ccat{D}_S(M)$ models this space.
\end{remark}

\subsubsection{}
The hypotheses on the factorizing basis are the following.
For a finite set $S$, denote by $\cat{E}_S$ the category of
$S$-labeled families of objects of $\cat{E}_1$ for which the tensor
product over $S$ is defined in $\cat{E}$.

\begin{hypothesis}\label{hy:monoidal-basis}
\mbox{}
\begin{itemize}
\item $\psi_1:=\psi\resto{\cat{E}_1}$ lands in $\Disk(M)$.
\item $\psi_1$ defines a (non-factorizing) effectively l-nice basis.
  (This is equivalent here to that
  $\psi_1\colon(\cat{E}_1)_{/U}\to\Disk(U)$ is an equivalence on the
  classifying spaces for every open $U\sub M$.
  See Remark \ref{rem:niceness-of-basis-by-disks}.
  $B\Disk(U)$ is equivalent to $U$.)
\item If a finite set $S$ consists of $1$ element, then $\cat{E}_S$ is
  the whole of $\cat{E}_1$.
\item For every finite set $S$, the square
  \begin{equation}\label{eq:pull-back-hypothesis}
  \begin{CD}
    \cat{E}_S@>>>\Dis_S(M)\\
    @V{\Tensor_S}VV@VV{\Disj_S}V\\
    \cat{E}@>\psi>>\Dis(M)
  \end{CD}    
  \end{equation}
  is Cartesian.
\end{itemize}
\end{hypothesis}

\begin{remark}
Considering the case where the finite set $S$ consists of $1$
element, we have a Cartesian square
\[\begin{CD}
\cat{E}_1@>\psi_1>>\Disk(M)\\
@VVV@VVV\\
\cat{E}@>\psi>>\Dis(M).
\end{CD}\]
In particular, the functor $\cat{E}_1\to\cat{E}$ is a full embedding.

Other $\cat{E}_S$ are (non-full) subcategories of $\cat{E}$.
\end{remark}

\begin{remark}
The consequence of the last condition of the above hypothesis which
will be actually used in the proof will be that 
for any object $D\in\ccat{D}_S(M)$, the square
\begin{equation}\label{eq:square-of-under-categories}
\begin{CD}
  (\cat{E}_S)_{D/}@>>>\Dis_S(M)_{D/}\\
  @V{\Disj_S}VV@VV{\Disj_S}V\\
  \cat{E}_{\Disj_SD/}@>\psi>>\Dis(M)_{\Disj_SD/}
\end{CD}
\end{equation}
is Cartesian.

This follows from the assumption since the assumption implies that the
square
\[\begin{CD}
(\cat{E}_S)_{E/}@>>>\Dis_S(M)_{E/}\\
@V{\Disj_S}VV@VV{\Disj_S}V\\
\cat{E}_{E/}@>\psi>>\Dis(M)_{E/}
\end{CD}\]
is Cartesian for every $E\in\ccat{D}(M)$, while the square
\[\begin{CD}
(\cat{E}_S)_{D/}@>>>\Dis_S(M)_{D/}\\
@V{\Disj_S}VV@VV{\Disj_S}V\\
(\cat{E}_S)_{\Disj_SD/}@>\psi>>\Dis_S(M)_{\Disj_SD/}
\end{CD}\]
is always Cartesian for every $D\in\ccat{D}_S(M)$.

In order to have that the square (\ref{eq:square-of-under-categories})
is Cartesian for every $D\in\ccat{D}_S(M)$, we do need the
full force of the assumption, since if we have that the map
$(\cat{E}_S)_{D/}\to[\cat{E}\times_{\Dis(M)}\Dis_S(M)]_{D/}$ is an
equivalence for every $D\in\ccat{D}_S(M)$, then the colimit of
this over all $D$ will be the original assumption.
\end{remark}

The following is a situation where the hypotheses are satisfied.
\begin{example}
\label{ex:free-disjoint-union}
Suppose given a (non-factorizing) effectively l-nice basis given by a
functor $\psi_1\colon\cat{E}_1\to\Open(M)$, $i\mapsto V_i$.
Then we can freely generate a symmetric, partially monoidal category
from $\cat{E}_1$ by using the partial monoidal structure of
$\Open(M)$.
Namely, we consider a category $\cat{E}$ whose objects are pairs
consisting of a finite set $S$ and a family $(i_s)_{s\in S}$ of
objects of $\cat{E}_1$ for which the open submanifolds
$V_{i_s}\sub M$ are pairwise disjoint.
The symmetric partial monoidal structure on $\cat{E}$ is defined in
the obvious way, and $\psi_1$ extends to a symmetric monoidal functor
$\cat{E}\to\Open(M)$, which we shall denote by $\psi$.

In this case, the underlying functor of $\psi$ defines an effectively
factorizing l-nice basis of $M$ at least if
$\psi_1$ (and so $\psi$ as well) is the inclusion of a full subposet.

If $\psi_1$ lands in $\Disk(M)$, then $\psi$ lands in $\Dis(M)$, and
the square (\ref{eq:pull-back-hypothesis}) is Cartesian by our
construction of the partial monoidal category $\cat{E}$.
\end{example}

\begin{example}
For an example of the previous example, we can take $\cat{E}_1$ to be
the full subposet of $\Open(M)$ consisting of open submanifolds
\emph{diffeomorphic} (rather than homeomorphic) to a disk.
In this case, $\cat{E}$ is the full subposet of $\Open(M)$ consisting
of open submanifolds diffeomorphic to the disjoint union of a finite
number of disks.
\end{example}

\begin{remark}\label{rem:factorization-basis-by-multifunctor}
$\cat{E}_1$ has a structure of a multicategory where for a finite set
$S$, the space of multimaps $i\to j$ for $i=(i_s)_{s\in S}$,
$i_s,j\in\cat{E}_1$, is non-empty if and only if $i\in\cat{E}_S$, and
in such a case,
\[
\Multimap_{\cat{E}_1}\bigl(i,j\bigr)\:=\:\Map_{\cat{E}}\bigl({\textstyle\Tensor_Si,j}\bigr).
\]
A symmetric monoidal functor on $\cat{E}$ restricts to an algebra on
$\cat{E}_1$, and this gives an equivalence of categories.
We may say that an algebra $A$ on $\cat{E}_1$ or equivalently, on
$\cat{E}$, is \kore{locally constant} if $A$ inverts all unary maps of
$\cat{E}_1$, and may denote the category of locally constant algebras
by $\Alg^\loc_{\cat{E}_1}(\cat{A})=\Alg^\loc_\cat{E}(\cat{A})$.

Our assumptions gives a functor $\cat{E}_1\to\Disk_1(M)$ of
multicategories, and Theorem \ref{thm:general-notion-of-algebra} may
be stated as that the induced functor
\[
\Alg_M(\cat{A})\longto\Alg^\loc_{\cat{E}_1}(\cat{A})
\]
is an equivalence.
\end{remark}

\begin{proof}[Proof of Theorem \ref{thm:general-notion-of-algebra}]
Necessity follows from the definition of local constancy and
Proposition \ref{prop:kan-extension-from-basis}.

For sufficiency, it suffices to prove that the given conditions on $A$
imply that the underlying functor of the restriction of $A$ to
$\Dis(M)$ extends to $\ccat{D}(M)$.
Indeed, once we have this, then Proposition
\ref{prop:cofinality-of-inversion} and the effective l-niceness of the
basis imply that, for every open $U\sub M$, the map
$\colim_{\cat{E}_{/U}}A\psi\to\colim_{\Dis(U)}A$ is an
equivalence, so $A$, which is assumed to be a left Kan extension from
$\cat{E}$, will in fact be a Kan extension from $\Dis(M)$.

In order to extend the underlying functor of $A\resto{\Dis(M)}$ to
$\ccat{D}(M)$, let us show that the right Kan extension $\extend{A}$
of $A\resto{\Dis(M)}$ to $\ccat{D}(M)$ coincides with $A$ on
$\Dis(M)$.
(For the solution for an issue here, see the remark after the proof.)
It actually suffices to show that the map $\extend{A}(jV)\to A(V)$, is
an equivalence for every $V$ in the factorizing basis, where
$j\colon\Dis(M)\to\ccat{D}(M)$ is the functor through which we are
comparing the two categories.
Indeed, if $D$ is an arbitrary object of $\Dis(M)$, and if we have
equivalences
\[
\colim_{\ccat{D}(D)}\extend{A}\xleftarrow{\equiv}\colim_{\cat{E}_{/D}}\extend{A}j\psi\xrightarrow{\equiv}\colim_{\cat{E}_{/D}}A\psi,
\]
then by the Kan extension assumption on $A$, we have that the map
$\extend{A}(D)\to A(D)$ is an equivalence.

In order to prove that the map
$\extend{A}(jV)=\lim_{\Dis(M)_{jV/}}A\to A(V)$ is an
equivalence, we shall first replace the shape of the diagram
over which this limit is taken, by a coinitial one.
Decompose $V$ into a disjoint union
$\Disj_{s\in S}\psi(i_s)$, $S$ a finite set, where $i_s\in\cat{E}_1$ so
$U_s:=\psi(i_s)=\psi_1(i_s)$ is a disk.
Then we shall prove that the functors
$(\cat{E}_S)_{jU/}\xrightarrow{\psi}\Dis_S(M)_{jU/}\into\Dis(M)_{jV/}$,
where $U=(U_s)_{s\in S}\in\Dis_S(M)$ (so $jV=\Disj_SjU$), are
coinitial.

The reason why the inclusion
$\Dis_S(M)_{jU/}\into\Dis(M)_{jV/}$ is coinitial is since this is
obviously a left adjoint.

In order to prove that the functor
$\psi\colon(\cat{E}_S)_{jU/}\to\Dis_S(M)_{jU/}$ is coinitial,
let us consider an object of $\Dis_S(M)_{jU/}$ which, as an object of
$\Dis(M)_{jV/}$, is given by the pair
consisting of an object $D$ of $\Dis(M)$ and a map $f\colon jV\to jD$
in $\ccat{D}(M)$.

Then, since we are requiring $f$ to be a map in $\ccat{D}_S(M)$, $D$
can be written as a disjoint union $\Disj_{s\in S}D_s$ of disks, where
the embedding part $g\colon V\into D$ of the data determining $f$,
embeds $U_s$ into $D_s$.
With this notation, it follows from definitions that the over
category $\left[(\cat{E}_S)_{jU/}\right]_{/(D,f)}$, which we want to
prove has contractible classifying space (here, $(D,f)$ is considered
as an object of $\Dis_S(M)_{jU/}$), is equivalent to
$\prod_{s\in S}(\cat{E}_1)_{/D_s,gU_s/}$, where we are
considering $D_s$ as an object of $\Disk(M)=\Dis_1(M)$, and $gU_s$ as
an object of $\ccat{D}_1(D_s)$, the full subcategory of $\ccat{D}(D_s)$
consisting of disks.

However, the functor
$j\psi_1\colon(\cat{E}_1)_{/D_s}\to\ccat{D}_1(D_s)$ is cofinal by the
assumption of effective l-niceness,
so we conclude that $(\cat{E}_1)_{/D_s,gU_s/}$ has contractible
classifying space, which implies that their product
$\left[(\cat{E}_S)_{jU/}\right]_{/(D,f)}$ also
has contractible classifying space.
This proves coinitiality of the functor
$\psi\colon(\cat{E}_S)_{jU/}\to\Dis_S(M)_{jU/}$.

It follows that the map
$\extend{A}(jV)\to\lim_{(\cat{E}_S)_{jU/}}A\psi$ is an
equivalence, so in order to conclude the proof, it suffices to show
that the map from this limit to $A(V)$ is an equivalence.

To analyse this limit, all the maps which appear in the diagram for
this limit are equivalences since they are induced from (a finite
family of) maps of $\cat{E}_1$, which $A\psi$ is assumed to invert.

It therefore suffices to show that the indexing category
$(\cat{E}_S)_{jU/}$ of the limit has contractible classifying space.
However, this follows from Lemma \ref{lem:cofinality-with-label} since
we have proved above that the functor
$\psi\colon(\cat{E}_S)_{jU/}\to\Dis_S(M)_{jU/}$ is coinitial.
\end{proof}

\begin{remark}
In the proof, we have used the right Kan extension of a functor taking
values in $\cat{A}$.
However, we do not need to assume existence of limits in $\cat{A}$ for
the validity of Theorem.
Indeed, our purpose for taking the Kan extension was to show that
the prealgebra $A$ was locally constant.
In order to prove this in the described method, $\cat{A}$ could be
fully embedded into a category which has all small limits (e.g., by
the Yoneda embedding), and the right Kan extension could be taken in
this larger category.
Note that the monoidal structure of $\cat{A}$ was not used in this
step of the proof.
\end{remark}

\begin{corollary}
\label{cor:descent-of-general-algebra}
Let $M$ be a manifold and let $\cover{V}$ be an effectively
factorizing l-nice basis of $M$ considered in
Theorem \ref{thm:general-notion-of-algebra}, equipped with all the
data, and satisfying all the assumptions.
Let $\cover{U}$ be another effectively factorizing l-nice basis of
$M$, given by a functor $\chi\colon\cat{C}\to\Open(M)$, $i\mapsto
U_i$.
Assume given a factorization $\psi=\chi\iota$, where
$\iota\colon\cat{E}\to\cat{C}$.

Then a prealgebra $A$ (not assumed to be locally constant) on $M$ is a
locally constant factorization algebra if (and only if) the following
are satisfied.
\begin{enumerate}\setcounter{enumi}{-1}
\item $A\psi$ inverts all morphisms of $\cat{E}_1$.
\item The functor $A\chi$ on $\cat{C}$ is a left Kan extension of
  its restriction $A\psi$ to $\cat{E}$ through $\iota$.
\item The underlying functor of $A$ is a left Kan extension of its
  restriction $A\chi$ to the basis $\cat{U}$.
\end{enumerate}
\end{corollary}

\begin{proof}
$A$ is a left Kan extension of its restriction to the basis
$\cover{V}$, so the previous theorem applies.
\end{proof}

\begin{example}\label{ex:compact-manifolds-as-basis}
Consider the following \emph{discrete} category $\manc$.
An object is a compact smooth manifold with boundary.
A map $\close{U}\to\close{V}$ is a smooth immersion of codimension
$0$ which restricts to an embedding $U\into V$, where $U$ and $V$ are
the interior of $\close{U}$ and $\close{V}$ respectively.
$\manc$ is a symmetric monoidal category under disjoint union.

Let $\close{M}$ be an object of this category, and let $M$ denote its
interior.

Then, in the corollary, we can take $\cat{U}$ to be given by the map
$\chi\colon(\manc)_{/\close{M}}\to\Open(M)$ of partial monoidal posets
sending $\close{U}\to\close{M}$ to its restriction $U\into M$ while
taking $\cat{E}_1$ to be the full subposet of $(\manc)_{/\close{M}}$
consisting of objects whose
source has the diffeomorphism type of the closed disk,
and $\cat{E}$ to be the symmetric partial monoidal category freely
generated by $\cat{E}_1$.
Here, we are considering $(\manc)_{/\close{M}}$ as a partial monoidal
category under unions which is disjoint in interiors, and are inducing
a structure of symmetric multicategory on $\cat{E}_1$ from this.
$\cat{E}$ is the full subposet of $(\manc)_{/\close{M}}$ generated from
$\cat{E}_1$ by the partial monoidal product.
\end{example}

In other words, a locally constant factorization algebra on this $M$
could be defined as a symmetric monoidal functor on
$(\manc)_{/\close{M}}$ whose underlying functor satisfies the first
two conditions of Corollary.
The original notion is recovered by taking the left Kan extension of
the underlying functor, to $\Open(M)$, which obtains a canonical
symmetric monoidal structure.

\subsubsection{}

\begin{proof}[Proof of Proposition \ref{prop:descent-of-algebras}]
Define $\cat{E}:=\laxcol_{i\in\cat{C}}\Dis(U_i)$, and let
$\iota\colon\cat{E}\to\cat{C}$ be the canonical projection and let
$\psi:=\chi\iota$.
Let $\cat{E}_1\sub\cat{E}$ be $\laxcol_{i\in\cat{C}}\Disk(U_i)$.

It suffices to check that
Corollary \ref{cor:descent-of-general-algebra} applies.

Firstly, $\psi\colon\cat{E}\to\Open(M)$ defines an effectively
factorizing l-nice basis of $M$ since for every open $U\sub M$, the
functors $\Dis(U_i)\to\cat{E}_{/i}$ for $i\in\cat{C}_{/U}$, the
functor $\colim_{i\in\cat{C}_{/U}}\cat{E}_{/i}\to\cat{E}_{/U}$, and so the
composite
$\colim_{i\in\cat{C}_{/U}}\Dis(U_i)\to\colim_{i\in\cat{C}_{/U}}\cat{E}_{/i}\to\cat{E}_{/U}$,
as well as the composite
$\colim_{i\in\cat{C}_{/U}}\Dis(U_i)\to\cat{E}_{/U}\to\ccat{D}(U)$ are
cofinal.

Similarly, $\psi_1$ defines an effectively l-nice basis.

Moreover, for a finite set $S$,
$\cat{E}_S=\laxcol_{i\in\cat{C}_{/U}}\Dis_S(U_i)$, and the rest of
Hypothesis \ref{hy:monoidal-basis} is satisfied.
\end{proof}

Finally, we prove the following from Theorem
\ref{thm:general-notion-of-algebra}.

\begin{theorem}\label{thm:descent-of-categories-of-algebras}
The presheaf $M\mapsto\Alg_M(\cat{A})$ of categories is a sheaf.
\end{theorem}
\begin{proof}
Let a cover of a manifold $M$ be given by $\cover{U}=(U_s)_{s\in S}$
where $S$ is an indexing set.
Let $\cat{C}:=\Simp_{/S}^\op$ be as in Example
\ref{ex:factorizing-cover}, and define $\chi\colon\cat{C}\to\Open(M)$
in the way described there.
We would like to prove that the restriction functor
\begin{equation}\label{eq:restricting-algebras}
\Alg_M(\cat{A})\longto\lim_{i\in\cat{C}}\Alg_{\chi(i)}(\cat{A})
\end{equation}
is an equivalence.
We shall construct an inverse.

For an open disk $D\in\Disk(M)$, define
\[
\cat{C}_D:=\{i\in\cat{C}\:\mid\:D\sub\chi(i)\}.
\]
Then this is either empty or has contractible classifying space.
Indeed, $\cat{C}_D=\Simp_{/S_D}^\op$, where $S_D:=\{s\in S\:\mid\:D\in
U_s\}$.

We plan to apply Theorem \ref{thm:general-notion-of-algebra} to the
following pair of basis.
Namely, define $\cat{E}_1$ to be the full subposet of $\Disk(M)$
consisting of disks $D$ such that $\cat{C}_D$ is non-empty.
This gives an l-nice basis of $M$.
Then define a factorizing l-nice basis $\cat{E}$ as in Example
\ref{ex:free-disjoint-union}.
The full inclusion $\psi\colon\cat{E}\into\Dis(M)$ is a map of
(symmetric) partial monoidal posets, and the pair
$\cat{E}_1\into\cat{E}$ of bases for the
topology of $M$ satisfies Hypothesis \ref{hy:monoidal-basis}.

Let $(A_i)_{i\in\cat{C}}\in\lim_{i\in\cat{C}}\Alg_{\chi(i)}$ be given.
Then define $B\colon\cat{E}_1\to\cat{A}$ by
$D\mapsto\lim_{i\in\cat{C}_D}A_i(D)$, so $B(D)$ is canonically
equivalent to $A_i(D)$ for any $i\in\cat{C}_D$.
Extend this uniquely to a
symmetric monoidal functor $B\colon\cat{E}\to\cat{A}$.
Then the left Kan extension of the underlying functor
$\cat{E}\to\cat{A}$ through $\psi\colon\cat{E}\to\Open(M)$ of $B$, has
a symmetric monoidal structure which makes it a locally constant
factorization algebra by Theorem \ref{thm:general-notion-of-algebra}.

Moreover, it is immediate that this is inverse to the restriction
functor \eqref{eq:restricting-algebras}.
\end{proof}

It follows that there is a notion of a locally constant factorization
algebra on an orbifold, and locally constant factorization algebras
can be pulled back along a local diffeomorphism (between orbifolds).

\section{Generalizations and applications}
\label{sec:generalize-apply}
\setcounter{subsection}{-1}
\setcounter{equation}{-1}
\subsection{Push-forward}
\label{sec:push-forward}
\setcounter{subsubsection}{-1}

We continue with the assumption stated in Section
\ref{sec:assumption-on-colimit}.

Theorem \ref{thm:factorizing-seifert-vk}
allows us to push forward an algebra along ``locally constant'' maps.

Given any map $p\colon X\to M$ of manifolds, the map
$p^{-1}\colon\Open(M)\to\Open(X)$ is symmetric monoidal.
It follows that any prealgebra on $X$ can be precomposed with $p^{-1}$
to give a prealgebra $p_*A$ on $M$.
Namely, we define $p_*A:=A\circ p^{-1}$.

We may ask when $p_*A$ is locally constant, whenever $A$ is a
locally constant factorization algebra.
It follows from Theorem \ref{thm:lurie-multilocalization} that a
sufficient condition is that $p$ is locally trivial in the sense that
over every component of $M$, it is the projection of a fibre bundle.
(Note that in this case, $p$ can be considered as giving a locally
constant family of manifolds parametrized by points of $M$.)

\begin{proposition}
If $p\colon X\to M$ is locally trivial, then for every locally
constant factorization algebra on $A$, the locally constant prealgebra
$p_*A$ is a factorization algebra.
\end{proposition}
\begin{proof}
Given any open submanifold $U$ of $M$, $p^{-1}$ maps the factorizing
l-nice cover $\Dis(U)$ of $U$ to a factorizing l-nice cover of
$p^{-1}U$.
Therefore, the result follows from
Theorem \ref{thm:factorizing-seifert-vk} applied to
$A\resto{p^{-1}U}$.
\end{proof}

Let us give the push-forward functoriality on the groupoid
of locally trivial maps.
By definition, this groupoid is modeled by a Kan complex $K_\dot$
whose $k$-simplex is a locally constant family over the standard
$k$-simplex of locally trivial maps.
In other words, a $k$-simplex is a map $p\colon X\times\simp^k\to
M\times\simp^k$ over $\simp^k$ which is locally trivial.

Note from Theorem \ref{thm:lurie-multilocalization} and
Corollary \ref{cor:isotopy-invariance}, that a locally constant
algebra $A$ on $X$ is functorial on the groupoid of open submanifolds
of $X$, which can be modeled by a Kan complex whose $k$-simplex is a
locally constant family over the standard $k$-simplex, of open
submanifolds.

Now let $p$ be an $k$-simplex of $K_\dot$.
Then for every open submanifold $U$ of $M$, the projection
$p^{-1}(U\times\simp^k)\to\simp^k$ gives a $k$-simplex of the
space of open submanifolds of $X$.
We obtain the desired functoriality of the push-forward immediately.

\subsection{Case of a higher target category}
\label{sec:higher-target-category}
\setcounter{subsubsection}{-1}

\subsubsection{}
 
A natural notion of a \emph{twisted} factorization algebra would be
the notion of an algebra taking values in a factorization
\emph{algebra} of categories, instead of in a symmetric monoidal
category.
A twisted algebra in this sense will turn out to be just a map between
certain algebras taking values in the Cartesian symmetric monoidal
category $\Cat$ of categories (of some limited size).
In particular, the space of twisted algebras is a part of the
structure of a category of $\Alg_M(\Cat)$.
However, in order to capture the structure of a \emph{category}
(rather than just a space) of twisted algebras, we need to take into
account the structure of a \emph{$2$-category} of $\Alg_M(\Cat)$,
coming from the $2$-category structure of $\Cat$.
We can consider algebras in a symmetric monoidal $2$-category
in general, and it is in fact natural to consider a symmetric
monoidal $n$-category for any $n\le\infty$.

\begin{definition}
Let $\cat{A}$ be a symmetric monoidal (infinity) infinity category.
Let $M$ be a manifold.
Then a \kore{locally constant factorization algebra} on $M$ in
$\cat{A}$ is an algebra in $\cat{A}$ over $\Elu_M$.
\end{definition}
If $\cat{A}$ is an $n$-category, then algebras in $\cat{A}$ form a
$n$-category.

The first thing to note is that the underlying $1$-category of the
$n$-category of
factorization algebras in $\cat{A}$ is just the category of algebras
in the underlying $1$-category of $\cat{A}$.

In order to understand the structure of the $n$-category of
factorization algebras, we would like to see that
Theorem \ref{thm:lurie-multilocalization} holds in this context, for
the $n$-categories of algebras.
It suffices to set $n=\infty$.

\begin{theorem}\label{thm:infinity-multilocalization}
Restriction through the morphism $\Disk(M)\to\Elu_M$ induces a fully
faithful functor between the (infinity) infinity categories of
algebras on these multicategories, valued in a symmetric monoidal
infinity category $\cat{A}$.
The essential image of the functor consists precisely of the locally
constant algebras on $\Disk(M)$.
\end{theorem}

In order to explain the proof this theorem, let us first review the
proof of Theorem \ref{thm:lurie-multilocalization}.
It follows from Theorem \ref{thm:lurie-approximation} and
Lemma \ref{lem:assumptions-satisfied} below.

The first theorem is as follows.
We shall comment on the undefined terms in it after we complete the
statement.
\begin{theorem}[A special case of Theorem 2.3.3.23 of
  \cite{higher-alg}]
  \label{thm:lurie-approximation}
Let $\cat{C}$ and $\cat{O}$ be multicategories, and assume that the
category of colours of $\cat{O}$ is a groupoid.
Let $f\colon\cat{C}\to\cat{O}$ be a morphism, and assume that it is a
\textbf{weak approximation}, and induces a homotopy equivalence on the
classifying spaces of the categories of colours.
Then, for every multicategory $\cat{A}$, the functor
\[
f^*\colon\Alg_\cat{O}(\cat{A})\longto\Alg_\cat{C}(\cat{A})
\]
induces an equivalence
$\Alg_\cat{O}(\cat{A})\equivto\Alg^\loc_\cat{C}(\cat{A})$,
where $\Alg^\loc$ here denotes the category of
\textbf{locally constant} algebras, and $\Alg$ denotes the category of
not necessarily locally constant algebras.
\end{theorem}

The \kore{local constancy} here means that the underlying functor of
the algebra inverts all (unary) morphisms between colours.
We do not need to explain the term ``\kore{weak approximation}'',
since we just quote the following.

\begin{lemma}[Lemma 5.2.4.10, 11 of \cite{higher-alg}]
\label{lem:assumptions-satisfied}
The assumptions on $f$ of Theorem \ref{thm:lurie-approximation} are
satisfied by the map $\Disk(M)\to\Elu_M$.
\end{lemma}

Thus, Theorem \ref{thm:lurie-multilocalization} extends to
Theorem \ref{thm:infinity-multilocalization} once we prove the
following.
\begin{proposition}\label{prop:infinity-lurie-inversion}
Let $\cat{C}$ and $\cat{O}$ be multicategories, and let
$f\colon\cat{C}\to\cat{O}$ be a morphism.
Assume that $f$ satisfies the \textbf{conclusion} of
Theorem \ref{thm:lurie-approximation} (for example,
by satisfying its assumptions).
Then the conclusion of the same theorem is true for any
\textbf{infinity} multicategory $\cat{A}$, instead of just
$1$-dimensional $\cat{A}$ (so an equivalence of (infinity)
\textbf{infinity} categories is the claimed conclusion).
\end{proposition}
\begin{proof}\setcounter{proofsec}{0}
\proofsec
It suffices to prove, for every finite $n$, the conclusion for
$n$-dimensional $\cat{A}$.
We shall do this by induction on $n$.
Since we know that the conclusion is true at the level of the
underlying $1$-categories, it suffices to prove that the functor $f^*$
is fully faithful.

Thus, suppose $n\ge 2$, and let $A,\:B\in\Alg_\cat{O}(\cat{A})$.
We need to recall the \emph{Day convolution}.
Namely, we construct an ($n-1$)-dimensional multicategory which we
shall denote by $\Map(A,B)$, equipped with a morphism to $\cat{O}$, so
that the ($n-1$)-dimensional category
$\Map_{\Alg_\cat{O}(\cat{A})}(A,B)$, is by definition, the fibre over
the universal $\cat{O}$-algebra $\id\colon\cat{O}\to\cat{O}$, of the
induced functor $\Alg_\cat{O}(\Map(A,B))\to\Alg_{\cat{O}}(\cat{O})$.
(This is actually a slightly modification of Day's original
construction, which captures lax, rather than genuine, morphisms of
algebras.)

An object of $\Map(A,B)$ is a pair $(x,\phi)$, where $x$ is an object
(or a ``colour'') in $\cat{O}$ ($x\in\cat{O}$), and $\phi\colon
A(x)\to B(x)$ in $\cat{A}$.
Given a family $(x,\phi)=((x_s,\phi_s))_{s\in S}$ of objects indexed
by a finite set $S$, and an object $(y,\psi)$, we define the
($n-2$)-category of multimaps by the equalizer diagram
\begin{equation*}
\Map((x,\phi),(y,\psi))\longto\Map_\cat{O}(x,y)\mathrel{\substack{\displaystyle\longto\\
    \displaystyle\longto}}\Map_\cat{A}(A(x),B(y)),
\end{equation*}
where the two maps equalized are the composites
\[
\Map_\cat{O}(x,y)\xlongrightarrow{B}\Map_\cat{A}(B(x),B(y))
\xlongrightarrow{\phi^*}\Map_\cat{A}(A(x),B(y))
\]
and
\begin{equation*}
\Map_\cat{O}(x,y)\xlongrightarrow{A}\Map_\cat{A}(A(x),A(y))
\xlongrightarrow{\psi_*}\Map_\cat{A}(A(x),B(y)).
\end{equation*}

For example, a multimap $(x,\phi)\to(y,\psi)$ is a pair
$(\theta,\alpha)$, where $\theta\colon x\to y$ in $\cat{O}$, and
$\alpha\colon B(\theta)\phi\equivto\psi A(\theta)$ in
$\Map(A(x),B(y))$, filling the square
\[\begin{CD}
A(x)@>\phi>>B(x)\\
@V{A(\theta)}VV@VV{B(\theta)}V\\
A(y)@>\psi>>B(y).
\end{CD}\]

Note that $\Map((x,\phi),(y,\psi))$ is indeed an ($n-2$)-category
since every fibre of the functor $\Map((x,\phi),(y,\psi))\to\Map(x,y)$
is ($n-2$)-dimensional, where the base is $0$-dimensional.

The functor $\Map(A,B)\to\cat{O}$ is given on objects by
$(x,\phi)\mapsto x$, and on multimaps by the projection
$\Map((x,\phi),(y,\psi))\to\Map(x,y)$.

\proofsec
We shall denote $\Map(f^*A,f^*B)$ by $\Map_\cat{C}(A,B)$.
The following is immediate from the definitions.
\begin{lemma}
The canonical square of multicategories
\[\begin{CD}
\Map_\cat{C}(A,B)@>>>\cat{C}\\
@VVV@VV{f}V\\
\Map(A,B)@>>>\cat{O}
\end{CD}\]
is Cartesian.
\end{lemma}

\proofsec
We shall continue with the proof of Proposition.
We have already seen that it suffices to prove that the functor
\[
f^*\colon\Map_{\Alg_\cat{O}(\cat{A})}(A,B)\longto\Map_{\Alg_\cat{C}(\cat{A})}(A,B)
\]
is an equivalence.
Lemma above implies that the square
\[\begin{CD}
\Alg_\cat{C}(\Map_\cat{C}(A,B))@>>>\Alg_\cat{C}(\cat{C})\\
@VVV@VV{f_*}V\\
\Alg_\cat{C}(\Map(A,B))@>>>\Alg_\cat{C}(\cat{O})
\end{CD}\]
is Cartesian.

From this, and the definition of
$\Map_{\Alg_\cat{C}(\cat{A})}(A,B)$, we obtain a Cartesian square
\[\begin{tikzcd}
\Map_{\Alg_\cat{C}(\cat{A})}(A,B)\arrow{d}\arrow{r}&\Alg_\cat{C}(\Map(A,B))\times_{\Alg_\cat{C}(\cat{O})}\Alg^\loc_\cat{C}(\cat{O})\arrow{d}\\
*\arrow{r}{\text{at }f}&\Alg^\loc_\cat{C}(\cat{O}).
\end{tikzcd}\]

From the inductive hypothesis, we also obtain a Cartesian square
\[\begin{tikzcd}
\Map_{\Alg_\cat{O}(\cat{A})}(A,B)\arrow{d}\arrow{r}&\Alg^\loc_\cat{C}(\Map(A,B))\arrow{d}\\
*\arrow{r}{\text{at }f}&\Alg^\loc_\cat{C}(\cat{O}).
\end{tikzcd}\]

It follows that the square
\[\begin{tikzcd}
\Map_{\Alg_\cat{O}(\cat{A})}(A,B)\arrow{d}[swap]{f^*}\arrow{r}&\Alg^\loc_\cat{C}(\Map(A,B))\arrow{d}\\
\Map_{\Alg_\cat{C}(\cat{A})}(A,B)\arrow{r}&\Alg_\cat{C}(\Map(A,B))\times_{\Alg_\cat{C}(\cat{O})}\Alg^\loc_\cat{C}(\cat{O})
\end{tikzcd}\]
is Cartesian.

Since in this square, the vertical map on the right is an inclusion
between full subcategories of $\Alg_\cat{C}(\Map(A,B))$, it follows
that the vertical map on the left identifies its source with the full
subcategory of its target consisting of those maps of algebras which,
as an algebra in $\Map(A,B)$, is locally constant.

The desired result now follows since the definition of a map of
algebras implies that every map of locally constant $\cat{C}$-algebras
is indeed locally constant in this sense.
\end{proof}

\begin{definition}
Let $\cat{A}$ be a symmetric monoidal infinity category.
Then a \kore{prealgebra} on a manifold $M$ in $\cat{A}$ is an algebra
over $\Open(M)$, in $\cat{A}$.
We say that a prealgebra $A$ is \kore{locally constant} if the
restriction of $A$ to a functor on $\Disk(M)$ is locally constant.
\end{definition}

Our descent results in the case the target category was a
$1$-category, described a locally constant factorization algebra as
a prealgebra satisfying various local constancy and descent properties
relative to a factorizing cover or basis satisfying certain
hypotheses.
Recall that these results depended on cofinality of functors to
$\ccat{D}(M)$.
Now we would like to see if same proofs work in the case where the
target category is now a symmetric monoidal infinity category.
For example, we have proved that
Theorem \ref{thm:lurie-multilocalization} holds in this context.

However, only this is a non-trivial result actually, and all of our
other proofs work without any change.
Namely, all of our descent results hold if our target is a symmetric
monoidal infinity category which (or equivalently, whose underlying
symmetric monoidal $1$-category) satisfies assumptions of
Section \ref{sec:assumption-on-colimit}.

\subsubsection{}
Finally, let us generalize Theorem \ref{thm:general-notion-of-algebra}
to twisted algebras.
Thus, let $M$ be a manifold, and let a basis for the topology of $M$
be given as in Theorem \ref{thm:general-notion-of-algebra}, by a
symmetric monoidal functor $\psi\colon\cat{E}\to\Open(M)$, $i\mapsto
V_i$, equipped with all the data, and satisfying all the assumptions.
In particular, $V_i\in\Dis(M)$ for every $i\in\cat{E}$.

\begin{lemma}\label{lem:universal-inversion-if-cofinal}
For $i\in\cat{E}$, if the composite
\begin{equation}\label{eq:universal-inversion-if-cofinal}
\cat{E}_{/i}\longto\cat{E}_{/V_i}\xlongrightarrow{\psi}\Dis(V_i)\longto\ccat{D}(V_i)
\end{equation} 
is cofinal, then this functor $\cat{E}_{/i}\to\ccat{D}(V_i)$ is
universal among the functors from $\cat{E}_{/i}$ which invert maps
which are inverted in $\ccat{D}(V_i)$.
Namely, for any category $\cat{C}$, the restriction through
\eqref{eq:universal-inversion-if-cofinal},
\[
\Fun(\ccat{D}(V_i),\cat{C})\longto\Fun(\cat{E}_{/i},\cat{C}),
\]
is fully faithful with image consisting of functors
$\cat{E}_{/i}\to\cat{C}$ which invert maps in $\cat{E}_{/i}$ inverted
in $\ccat{D}(V_i)$.
\end{lemma} 

\begin{remark}\label{rem:condition-for-cofinality}
From Remark \ref{rem:niceness-of-basis-by-disks}, the assumption of
the cofinality follows if the first map
$\cat{E}_{/i}\to\cat{E}_{/V_i}$ of the composition
\eqref{eq:universal-inversion-if-cofinal} is cofinal, e.g., by being
an equivalence.
\end{remark} 

\begin{proof}[Proof of Lemma]
In order to show that the restriction functor is fully faithful, we
may first embed $\cat{C}$ by a fully faithful functor (e.g.~the Yoneda
embedding) into a category which has all small limits in it, and
show that the restriction functor is fully faithful for this larger target
category, in place of $\cat{C}$.
Therefore, we do not lose generality by assuming that $\cat{C}$
has all small limits in it, as we shall do.

In this case, an argument similar to the proof of Theorem
\ref{thm:general-notion-of-algebra} implies that the restriction
functor is the inclusion of a right localization of
$\Fun(\cat{E}_{/i},\cat{C})$.
Namely, if $U\in\ccat{D}(V_i)$ is of the form $\Disj_{s\in S}D_s$ for a
family $D=(D_s)_{s\in S}$ of disjoint disks indexed by a finite set
$S$, so $D\in\Dis_S(V_i)$, then we have
$\psi_S\colon(\cat{E}_S)_{/i}\to\Dis_S(V_i)$, and the resulting functor
$((\cat{E}_S)_{/i})_{D/}\to(\cat{E}_{/i})_{U/}$ is coinitial since it
has a right adjoint.
It follows that the right Kan extension of a functor
$F\in\Fun(\cat{E}_{/i},\cat{C})$ to $\ccat{D}(V_i)$ associates to $U$
the limit $\lim_{((\cat{E}_S)_{/i})_{D/}}F$.
The claim follows immediately from this, so we have proved the fully
faithfulness of the restriction functor.

The identification of the image of the embedding is then also
immediate.
\end{proof}

Let $M$ be a manifold, and let $\Dis_M$ denote $\Dis$ considered as
an algebra of categories on $\Disk(M)$.
Then in the $2$-category $\Alg_{\Disk(M)}(\Cat)$ of (not necessarily
locally constant) algebras of categories on $\Disk(M)$, $\Dis_M$
corepresents the functor $\cat{A}\mapsto\Alg_{\Disk(M)}(\cat{A})$.

Similarly, let $\ccat{D}_M$ denote $\ccat{D}$ as a (locally constant)
algebra on $\Disk(M)$.
The obvious functor $\Dis\to\ccat{D}$ is a map of algebras.
We obtain the following by applying Lemma to the basis $\Dis(M)$ for
the topology of $M$.

\begin{corollary}\label{cor:representing-locally-constant-algebras}
Let $M$ be a manifold, and let $\cat{A}$ be an algebra of categories
on $\Disk(M)$.
Then the restriction functor
\[
\Map_{\Alg_{\Disk(M)}}(\ccat{D}_M,\cat{A})\longto\Map_{\Alg_{\Disk(M)}}(\Dis_M,\cat{A})=\Alg_{\Disk(M)}(\cat{A})
\]
through the map $\Dis\to\ccat{D}$ is fully faithful, and the image
consists precisely of the locally constant algebras in $\cat{A}$.
\end{corollary}

More generally, in our current situation as in Theorem
\ref{thm:general-notion-of-algebra}, let $\ccat{D}_{\cat{E}_1}$ denote
the restriction of $\ccat{D}_M$ through the functor
$\psi\colon\cat{E}_1\to\Disk(M)$ of multicategories (see Remark
\ref{rem:factorization-basis-by-multifunctor}).
Then Lemma \ref{lem:universal-inversion-if-cofinal} implies that if
the functor \eqref{eq:universal-inversion-if-cofinal} is cofinal for
every $i\in\cat{E}_1$, then $\ccat{D}_{\cat{E}_1}$ corepresents the
functor $\cat{A}\mapsto\Alg^\loc_{\cat{E}_1}(\cat{A})$ on
$\Alg_{\cat{E}_1}(\Cat)$.
As a consequence, we obtain the following, twisted version of Theorem
\ref{thm:general-notion-of-algebra}, from the $2$-categorical
generalization of Theorem \ref{thm:general-notion-of-algebra} (in the
case of the target $2$-category $\Cat$).

\begin{theorem}\label{thm:twisted-general-algebra}
Let $M$ be a manifold, and let $\cat{A}$ be a locally constant
factorization algebra of categories on $M$.
Then for a basis for the topology of $M$ as in Theorem
\ref{thm:general-notion-of-algebra}, if the functor
\eqref{eq:universal-inversion-if-cofinal} is cofinal for every
$i\in\cat{E}_1$, then the restriction functor
\[
\Alg_M(\cat{A})\longto\Alg^\loc_{\cat{E}_1}(\cat{A})
\]
is an equivalence.
\end{theorem}

\begin{remark}\label{rem:assumption-for-twisted-algebra-theorem}
See Remark \ref{rem:condition-for-cofinality} for a sufficient
condition for the assumption here to be satisfied.
\end{remark}

\subsection{(Twisted) algebras on a (twisted) product}
\setcounter{subsubsection}{-1}

\subsubsection{}
 
We shall illustrate applications of
Theorem \ref{thm:general-notion-of-algebra} and its generalization
Theorem \ref{thm:twisted-general-algebra}.

Fix a target symmetric monoidal category satisfying the assumptions of
Section \ref{sec:assumption-on-colimit}, and drop the name of this
category from the notation.

\begin{theorem}\label{thm:prod-formula}
Let $B$, $F$ be manifolds.
Then, the restriction functor
\[
\Alg_{F\times B}\longto\Alg_B(\Alg_F)
\]
is an equivalence.
\end{theorem}
\begin{proof}\setcounter{proofsec}{0}
\proofsec
Note that the category $\Alg_F$ has sifted colimits, and they are
preserved by the tensor product (since these are the same colimits and
tensor product on the underlying objects).

We would like to use Theorem \ref{thm:general-notion-of-algebra} on
$M:=F\times B$.
For this purpose, we consider the following basis for the topology of
$M$.

The basis will be indexed by the symmetric partially monoidal category
$\cat{E}$ to be defined as follows.
The underlying category of $\cat{E}$ will be as follows.
Its objects are any object $D$ of $\Dis(M)$ for which there
exists objects $D'$ of $\Dis(B)$ and $D''$ of $\Dis(F)$, such that any
component of $D$ is a component of $D'\times D''\sub M$.

Morphisms in $\cat{E}$ shall be just inclusions, so it is a full
subposet of $\Dis(M)$.
We denote the inclusion by $\psi\colon\cat{E}\into\Dis(M)$.
Note that this determines a factorizing l-nice (and hence effectively
factorizing l-nice by Proposition \ref{prop:van-kampen-cofinality})
basis of $M$.

The partial monoidal structure on $\cat{E}$ will be defined as
follows.
Namely, for any finite set $S$, let $\Dis(M)^{(S)}$ denote the full
subposet of the Cartesian product $\Dis(M)^S$ on which the disjoint
union operation to $\Dis(M)$ is defined.
Then we define the poset $\cat{E}^{(S)}$ by the Cartesian square
\begin{equation}\label{eq:domain-for-operation}
\begin{tikzcd}
\cat{E}^{(S)}\arrow{d}\arrow[hook]{r}&\Dis(M)^{(S)}\arrow{d}{\Disj_S}\\
\cat{E}\arrow[hook]{r}{\psi}&\Dis(M).
\end{tikzcd}
\end{equation} 
It is canonically a full subposet of $\cat{E}^S$, and we let it be
the domain of definition of the $S$-fold monoidal operation of
$\cat{E}$, where the operation is defined to be the left vertical map
on the square \eqref{eq:domain-for-operation}.
Since $\cat{E}$ is a poset, this determines a partial monoidal
structure on $\cat{E}$.

We define the full subposet $\cat{E}_1$ of $\cat{E}$ to be the
intersection $\cat{E}\intersect\Disk(M)$ taken in $\Dis(M)$.
(As a full subposet of $\Disk(M)$, $\cat{E}_1$ is
$\Disk(F)\times\Disk(B)$.)

For this factorizing l-nice basis of $M$, equipped with auxiliary
data required for Theorem \ref{thm:general-notion-of-algebra}, we
would like to verify that the Hypothesis \ref{hy:monoidal-basis} is
satisfied.
All but the hypothesis that
$\psi_1:=\psi\resto{\cat{E}_1}\colon\cat{E}_1\to\Open(M)$ determines
an effectively l-nice basis, are easily verified from the construction.
This remaining hypothesis follows from Lurie's generalized
Seifert--van Kampen Theorem, since it is immediate
to see that $\psi_1$ determines an l-nice basis.

\proofsec
Now Theorem \ref{thm:general-notion-of-algebra} implies that the
restriction functor $\Alg_M\to\Alg^\loc_\cat{E}$ is an equivalence,
where the target is the category of algebras on $\cat{E}$ \emph{which
  is locally constant} with respect to $\cat{E}_1$ in the sense that the
maps in $\cat{E}_1$ are all inverted.

However, the restriction functor $\Alg^\loc_\cat{E}\to\Alg_B(\Alg_F)$
is nearly tautologically (namely, up to introduction and elimination
of the unit objects and the unit operations as necessary) an
equivalence.

This completes the proof.
\end{proof}

For example, a locally constant factorization algebra on $\R^2$ is the
same as an associative algebra in the category of associative
algebras since a locally constant factorization algebra on
$\R^1$ can be directly seen to be the same as an associative algebra.

Inductively, a locally constant factorization algebra on $\R^n$ is an
iterated associative algebra object.

\begin{remark}\label{rem:dunn}
A product manifold $M=B\times F$ has another interesting factorizing
basis.
Namely, there is a factoring basis of $M$ consisting of the disjoint
unions of disks in $M$ of the form $D'\times D''$ for disks $D'$ in
$B$ and $D''$ in $F$.
As observed in Example \ref{ex:free-disjoint-union},
Theorem \ref{thm:general-notion-of-algebra} applies to the factorizing
basis freely generated by this basis.
The result we obtain is another description of the category
$\Alg_M$, namely as the category of `locally constant' algebras on
this factorizing basis.

Iterating this, one finds a description of the category of locally
constant algebras on $\R^n$ which identifies it essentially with
the category of algebras over Boardman--Vogt's ``little cubes''
\cite{bv}.
Therefore, Theorem \ref{thm:prod-formula} can be considered as a
generalization of a theorem of Dunn \cite{dunn}.
\end{remark}

\begin{remark}
Dunn's theorem actually identifies the $E_n$-operad with the $n$-fold
tensor product of the $E_1$-operad.
In particular, unlike our theorem in the case of $\R^n$, the target
category of the algebras need not satisfy our assumptions on sifted
colimits.
\end{remark}

\subsubsection{}
Theorem \ref{thm:prod-formula} identified the algebras on a product
manifold.
A product of manifolds has a twisted version, namely, a
fibre bundle.
Accordingly there is a generalization of
Theorem \ref{thm:prod-formula} which holds for a fibre bundle.
Let us formulate and prove it.

Let $p\colon E\to B$ be a smooth fibre bundle over a smooth base
manifold (i.e., a map with ``locally constant'' fibres).
Then we construct a locally constant algebra $\Alg_{E/B}$ of
categories on $B$ as follows.
Given an open disk $D\sub B$, let $\Alg_{E/B}(D)$ be
the category $\Alg_{E_x}$ for the unique (up to a contractible space
of choices) point $x\in D$.
Note that the manifold $E_x$ is unambiguously specified by $D$ in the
infinity groupoid of manifolds where the spaces of morphisms are
the spaces of diffeomorphisms.

An inclusion $D\into D'$ of disks in $B$ induces an equivalence
$\Alg_{E/B}(D)\to\Alg_{E/B}(D')$ of symmetric monoidal categories
(specified uniquely up to a contractible space of choices).
This association becomes an algebra on $\Disk(B)$ since given a
disjoint inclusion $\Disj_{s\in S}D_s\into D'$ of disks in $B$, then
we have a functor $\prod_{s\in S}\Alg_{E/B}(D_s)\to\Alg_{E/B}(D')$
defined as (the underlying functor of) the unique symmetric monoidal
functor extending the symmetric monoidal functors
$\Alg_{E/B}(D_s)\to\Alg_{E/B}(D')$.
This defines $\Alg_{E/B}$ as a locally constant algebra of
categories on $B$.

\subsubsection{}
Alternatively, given a disk $D\sub B$, consider a trivialization of
$p$ over $D$.
If $F$ is the typical fibre in the trivialization, then we define
$\Alg_{E/B}(D)$ to be $\Alg_{F}$.
A different trivialization with typical fibre $F'$ specify a
diffeomorphism $F\equivto F'$ uniquely up to a contractible space of
choices (we will have a family of diffeomorphisms parametrized by
$D$).
Moreover the specified (family of) diffeomorphisms satisfy the
cocycle condition.
This eliminates the ambiguity of $\Alg_{E/B}(D)$.

With a trivialization as above fixed, we shall call $F$ the
\kore{fibre over $D$} of $p$.

In this approach, the algebra structure of $\Alg_{E/B}$ is given by
the symmetric monoidal structure of $\Alg_F$.
Namely, if a disjoint inclusion $\Disj_{s\in S}D_s\into D'$ of disks
in $B$ is given, then a trivialization of $p$ over $D'$ restricts to a
trivialization over each $D_s$, and then all $\Alg_{E/B}(D_s)$ get
canonically identified with $\Alg_F=\Alg_{E/B}(D')$, where $F$ is the
fibre over $D'$ of $p$ with respect to the chosen trivialization, so
the monoidal operation $\Tensor_S\colon\Alg_F^S\to\Alg_F$ becomes the
desired operation
\[
\prod_{s\in S}\Alg_{E/B}(D_s)\longto\Alg_{E/B}(D').
\]

This is compatible with the structure of symmetric multicategory on
$\Disk(B)$ since restriction of trivializations clearly is.

\subsubsection{}
The relation of this approach to the previous approach is that a
trivialization of $p$ over a disk $D$ in $B$, gives an identification
of $E_x$, $x\in D$, with the fibre of $p$ over $D$.

\subsubsection{}
Next, we shall construct the ``restriction'' functor
$\Alg_E\to\Alg_B(\Alg_{E/B})$.
Given an algebra $A$ on $E$, we shall associate to it an object of
$\Alg_B(\Alg_{E/B})$ denoted by $A_{E/B}$ as follows.

Given an open disk $D\sub B$, we pick a trivialization of $p$ over
$D$, and denote by $q$ the projection $p^{-1}D\to F$ with respect to
the trivialization, where $F$ is the fibre of $p$ over $D$ (with
respect to the trivialization).
Then we define $A_{E/B}(D):=q_*i^*A\in\Alg_F=\Alg_{E/B}(D)$,
where $i\colon p^{-1}D\into E$ is the inclusion.

We need to check the well-definedness of this construction.
Recall that we identified different models of the fibre of $p$ over
$D$ by comparing the family $F\times D$ over $D$, for any one model
$F$, with the family $p^{-1}D$, by the trivialization making $F$ be a
model for the fibre over $D$.

Taking this into account, it is easy to see that, in order to
eliminate the ambiguity of the construction, it suffices to give
a path between the maps $q\times D\colon p^{-1}D\times D\to F\times
D$ and $p^{-1}D\times D\xrightarrow{\pr}p^{-1}D\equivwith F\times D$,
through locally trivial maps (maps with locally constant fibres)
$p^{-1}D\times D\to F\times D$.
Using the trivialization $p^{-1}D\equivwith F\times D$ again, this is
equivalent to giving a path between the two projections $F\times
D\times D\to F\times D$ through locally trivial maps.

We may instead choose a path between the two projections $D^2\to D$,
through locally trivial maps.
We pick an embedding of $D$ into a vector space as an open
\emph{convex} subdisk, which does not add more information than a
choice of a point from a contractible space.
Then we have a path of locally trivial maps $D^2\to D$.
\[
(x,y)\longmapsto x+t(y-x),\;0\le t\le 1.
\]
This clearly comes as a family over the said contractible space.

Let us now equip this association $D\mapsto q_*i^*A$ with a structure
of an algebra over $\Disk(B)$.
The construction is similar to the construction of the algebra
structure of $\Alg_{E/B}$, which we have made before.
Namely, if we are given an inclusion $D\into D'$ in $B$, where $D$ is
a disjoint union of disks, then a trivialization of $p$ over $D'$
restricts to a trivialization of $p$ over $D$, and thus we can try to
construct the desired map $A_{E/B}(D)\to A_{E/B}(D')$ as
$A[(q\resto{p^{-1}D})^{-1}(U)=q^{-1}(U)\cap p^{-1}D\into q^{-1}(U)]$
for disks $U\sub F$, $F$ the fibre over $D'$.

It remains to check that this construction is compatible with the
construction we have made to eliminate the ambiguity for the
association $D\mapsto A_{E/B}(D)$.
Again, assuming that $D'$ is an open convex subdisk of a vector space,
it does no harm to restrict $D$ to ones which are disjoint unions of
open \emph{convex} subdisks of $D'$ (convexity in the same vector
space).

Then the path of locally constant maps $(D')^2\to D'$ given above
restricts to a similar path on each component of $D$.
This verifies the compatibility of the constructions.

It follows that the construction above indeed defines an algebra
$A_{E/B}\in\Alg_B(\Alg_{E/B})$.

\begin{proposition}\label{prop:twisted-product-formula}
Let $p\colon E\to B$ be a smooth fibre bundle as above.
Then the restriction functor
\[
\Alg_E\longto\Alg_B(\Alg_{E/B})
\]
is an equivalence.
\end{proposition}
\begin{proof}
The functor can be written as
\[
\lim_{D\in\Dis(B)}\Alg_{p^{-1}D}\longto\lim_{D\in\Dis(B)}\Alg_D(\Alg_{p^{-1}D/D}).
\]
Indeed, we can apply
Theorem \ref{thm:basic-descent-of-category-of-algebras} to the source,
and the target is this limit essentially by definition.

The given functor is the limit of the restriction functors on
$D\in\Dis(B)$.

However, on each $D$, the restriction functor can be identified with
that in Theorem \ref{thm:prod-formula} by using the decomposition
$p^{-1}D=F\times D$, where $F$ is the fibre of $p$ over $D$.
Therefore it is an equivalence by the assertion of the theorem.

It follows that the twisted version of the restriction functor is also
an equivalence.
\end{proof}

\begin{remark}\label{rem:higher-target-category}
From the discussions of Section \ref{sec:higher-target-category},
Proposition holds for a higher target category by the same proof.
If the target is an $n$-category, then the proposition states that we
have an equivalence of $n$-categories of algebras.
In the following, we shall use the $2$-category case.
\end{remark}

\subsubsection{}
There is a natural further generalization of this.
Namely, the algebra $\Alg_{E/B}$ can
be constructed when the algebra on $E$ is twisted.
That is, let $\cat{A}$ be a locally constant (pre-)algebra on $E$ of
categories.
Then, for a disk $D\sub B$, define the category
\[
\Alg_{E/B}(\cat{A})(D):=\Alg_F(\cat{A}_{E/B}(D)),
\]
where $\cat{A}_{E/B}\in\Alg_B(\Alg_{E/B}(\Cat))$ (where $\Cat$ denotes
the $2$-category of categories in which $\cat{A}$ is taking values) is
the restriction of $\cat{A}$ as in the previous proposition, and $F$
is the fibre of $p$ over $D$, so
$\cat{A}_{E/B}(D)\in\Alg_{E/B}(\Cat)(D)=\Alg_F(\Cat)$.

Moreover, a restriction functor
\begin{equation}\label{eq:twisted-restriction}
\Alg_E(\cat{A})\longto\Alg_B(\Alg_{E/B}(\cat{A}))
\end{equation}
can be defined by $A\mapsto A_{E/B}$, where
$A_{E/B}\in\Alg_B(\Alg_{E/B}(\cat{A}))$ associates to a disk $D\sub
B$, the object
$q_*i^*A\in\Alg_F(q_*i^*\cat{A})=\Alg_{E/B}(\cat{A})(D)$.
The algebra structure is exactly as before.

\begin{theorem}\label{thm:twisted}
For a locally constant (pre-)algebra $\cat{A}$ on $E$ of categories,
the restriction functor \eqref{eq:twisted-restriction} is an
equivalence.
\end{theorem}

Let us first establish this in the case where the fibre bundle is
trivial.
A global choice of a trivialization leads to simplification of the
constructions as well.

\begin{lemma}\label{lem:non-twisted-product-for-twisted-algebra}
Let $B$, $F$ be manifolds, and let $\cat{A}$ be an object of
$\Alg_B(\Alg_F(\Cat))$, or equivalently, a locally constant
algebra of categories on $F\times B$, by
Theorem~\ref{thm:prod-formula}.

Then, the restriction functor
\[
\Alg_{F\times B}(\cat{A})\longto\Alg_B(\Alg_F(\cat{A})),
\]
where $\Alg_F(\cat{A})$ is a locally constant algebra of categories on
$B$, defined by $\Alg_F(\cat{A})(D):=\Alg_F(\cat{A}(D))$,
is an equivalence.
\end{lemma}
\begin{proof}
Similar to the proof of Theorem \ref{thm:prod-formula}.
One simply notes that Theorem \ref{thm:twisted-general-algebra}
applies here instead of Theorem \ref{thm:general-notion-of-algebra}.
See Remark \ref{rem:assumption-for-twisted-algebra-theorem}.
\end{proof}

\begin{proof}[Proof of Theorem \ref{thm:twisted}]
The $2$-categorical generalization of Theorem
\ref{thm:basic-descent-of-category-of-algebras} implies that the
restriction functor
$\Alg_E(\Cat)\to\lim_{D\in\Dis(B)}\Alg_{p^{-1}D}(\Cat)$ is an
equivalence of $2$-categories.
From this, one obtains that the restriction functor
\[
\Alg_E(\cat{A})\longto\lim_{D\in\Dis(B)}\Alg_{p^{-1}D}(\cat{A})
\]
is an equivalence.

Similarly, one would like to show that the restriction functor
\[
\Alg_B(\Alg_{E/B}(\cat{A}))\longto\lim_{D\in\Dis(B)}\Alg_D(\Alg_{p^{-1}D/D}(\cat{A}))
\]
is an equivalence.
However, since it is easy to verify from the definitions, that the
restriction of $\Alg_{E/B}(\cat{A})$ to $D\sub B$ is
$\Alg_{p^{-1}D/D}(\cat{A})$, the equivalence also follows from Theorem
\ref{thm:basic-descent-of-category-of-algebras}.

By the naturality of the restriction functor, we have reduced the
statement to the case where the base is a disjoint union of disks.
In this case the fibre bundle is trivial on each component, and the
statement follows from Lemma.
\end{proof}

\begin{remark}
The results of this section depended only on our descent results from
Section \ref{sec:f-alg}.
Therefore, by what have been seen in the previous section, all the
results of this section have a version in which the target category is
infinite dimensional, and we get an equivalence of \emph{infinity}
categories of algebras.
\end{remark}

\end{document}